\documentclass[a4paper,11pt]{amsart}

\usepackage{amssymb,xspace,amscd,yhmath,mathdots,txfonts,
mathrsfs}

\usepackage[matrix,arrow,curve]{xy}\CompileMatrices

\usepackage{hyperref}

\usepackage{yfonts}[1998/10/03]

\DeclareMathAlphabet{\mathpzc}{OT1}{pzc}{m}{it}

\numberwithin{equation}{section}

\theoremstyle{plain}

\newtheorem{thm}{Theorem}[section]

\newtheorem{lem}[thm]{Lemma}

\newtheorem{cor}[thm]{Corollary}

\newtheorem{prop}[thm]{Proposition}

\theoremstyle{definition}

\newtheorem{defn}{Definition}[section]
\newtheorem{exam}[thm]{Example}
\newtheorem{ntz}{Notation}[section]

\newtheorem{rmk}[thm]{Remark}
\newcommand\Lf{\mathbf{L}}
\newcommand\Ff{\mathbf{F}}
\newcommand\et{\mathfrak{e}}
\DeclareMathAlphabet{\mathpzc}{OT1}{pzc}{m}{it}

\DeclareMathOperator{\R}{{\mathbb{R}}}
\DeclareMathOperator{\Ub}{\mathbf{U}}
\DeclareMathOperator{\cC}{{\mathcal{C}}}

\DeclareMathOperator{\kt}{{\kappaup}}

\DeclareMathOperator{\Li}{\mathfrak{L}}

\DeclareMathOperator{\nil}{\mathrm{nil}}
\DeclareMathOperator{\rad}{\mathrm{rad}}
\DeclareMathOperator{\radn}{\mathrm{rad}_n}

\DeclareMathOperator{\Gf}{\mathbf{G}}
\DeclareMathOperator{\nr}{\mathfrak{n}_{{\kt}}}

\DeclareMathOperator{\ad}{\mathrm{ad}}
\DeclareMathOperator{\Ad}{\mathrm{Ad}}
\DeclareMathOperator{\Lie}{\mathrm{Lie}}

\DeclareMathOperator{\Ct}{\mathrm{C}}

\DeclareMathOperator{\zt}{\mathfrak{z}}
\DeclareMathOperator{\st}{\mathfrak{s}}

\DeclareMathOperator{\at}{\mathfrak{a}}
\DeclareMathOperator{\pt}{\mathfrak{p}}
\DeclareMathOperator{\ft}{\mathfrak{f}}

\DeclareMathOperator{\re}{{\mathrm{Re}}}

\DeclareMathOperator{\C}{\mathbb{C}}

\DeclareMathOperator{\dist}{{\mathrm{dist}}}
\newcommand\Kf{\mathbf{K}}

\newcommand\Qf{\mathbf{Q}}

\newcommand\GL{\mathbf{GL}}

\newcommand\wt{\mathfrak{w}}

\newcommand\qt{\mathfrak{q}}

\newcommand\Rad{\mathpzc{R}}
\newcommand\hg{\mathfrak{h}}
\newcommand\If{\mathbf{V}}
\newcommand\vt{\mathfrak{v}}
\newcommand\eg{\mathfrak{e}}

\newcommand\Vf{\mathbf{V}}
\newcommand\lt{\textswab{l}}
\newcommand\mt{\textswab{m}}
\newcommand\Zf{\mathpzc{Z}}

 \newcommand\op{\mathrm{p}_0}
\newcommand\Wf{\mathbf{W}}
\newcommand\SL{\mathbf{SL}}
\newcommand\SU{\mathbf{SU}}
\newcommand\CP{\mathbb{CP}}
\newcommand\Gr{\mathpzc{Gr}}
\newcommand\Fi{\mathpzc{F}}

\newcommand\Xg{\mathfrak{X}}

\newcommand\Sb{\mathbf{S}}
\newcommand\Id{\mathrm{I}}

\newcommand\Ci{\mathcal{C}^\infty}

\newcommand\frm{\mathbf{b}}
\newcommand\Nf{\mathbf{N}}
\newcommand\nt{\mathfrak{n}}

\newcommand\ut{\mathfrak{u}}

\newcommand\po{\mathfrak{p}_0}
\newcommand\Pf{\mathpzc{P}}
\newcommand\Pfo{\mathpzc{P}_0}
\newcommand\Mf{\mathpzc{M}}

\newcommand\trac{\mathrm{trace}}

\newcommand{\Ef}{\mathbf{E}}

\newcommand\slt{\mathfrak{sl}}
\newcommand\su{\mathfrak{su}}
\newcommand\Jf{\mathpzc{J}}

\newcommand\Pp{\mathfrak{P}}
\newcommand\td{\mathfrak{t}}
\newcommand\Cf{\mathbf{C}}
\newcommand\Ml{\mathpzc{M}}
\newcommand\cg{\mathfrak{c}}
\newcommand\Af{\mathbf{A}}

\newcommand\HNR{$\mathrm{HNR}$}
\newcommand\kll{\mathpzc{k}}

\newcommand\Tb{\mathbf{T}}

\newcommand\trasp{{{^T}\!}}

   \def\DHLhksqrt#1#2{\setbox0=\hbox{$#1\sqrt{#2\,}$}\dimen0=\ht0
     \advance\dimen0-0.2\ht0
     \setbox2=\hbox{\vrule height\ht0 depth -\dimen0}%
     {\box0\lower0.4pt\box2}}

\title[Mostow's fibration and $CR$ manifolds]{Mostow's Fibration for 
canonical embeddings of compact homogeneous
$CR$ manifolds}

\author[S.Marini]{Stefano Marini}
\address{S.~Marini: Dipartimento di Matematica e Fisica, 
III Universit\`a di Roma, 
Largo San Leonardo Murialdo, 1
00146 Roma (Italy) }
\email{marinistefano86@gmail.com}
\author[M.~Nacinovich]{Mauro Nacinovich}
\address{M.Nacinovich:
Dipartimento di Matematica\\ II Universit\`a di Roma
``Tor Ver\-ga\-ta''\\ Via della Ricerca Scientifica\\ 00133 Roma
(Italy)}
\email{nacinovi@mat.uniroma2.it}

\date\today

\subjclass[2000]{Primary: 32V30
Secondary: 32V25, 32V35,  
 53C30}
\keywords{homogeneous $CR$-manifold, $CR$-embedding, 
Mostow fibration, Matsuki duality, 
tangential Cauchy-Riemann complex, Dolbeault cohomology}

\begin{document}
\begin{abstract}
We define a class of compact homogeneous $CR$ manifolds 
which are bases of Mostow fibrations having total spaces equal to their canonical
complex realizations and Hermitian fibers. This is used to establish isomorphisms
between their tangential Cauchy-Riemann cohomology groups and the corresponding
Dolbeault cohomology groups of the embeddings.
\end{abstract}
\maketitle

\tableofcontents
\section{Introduction and preliminaries}
The aim of this paper
is to investigate relations between the cohomology groups of the
tangential Cauchy Riemann complexes  
of $\nt$-reductive compact homogeneous $CR$ manifolds and the corresponding Dolbeault
cohomology groups  of their canonical embeddings. The class of 
$\nt$-reductive compact homogeneous $CR$
manifolds
was introduced  
in \cite{AMN2013}: its objects are the minimal orbits, in homogeneous spaces of 
reductive complex groups, of their compact forms. 
\par 
Results on the cohomology of the tangential $CR$ complexes on 
general compact $CR$ manifolds 
of arbitrary codimension were
obtained in \cite{40} (see also \cite{BHN2016}), under suitable $r$-pseu\-do\-con\-cavity conditions,
involving their \textit{scalar Levi forms}, 
that were first introduced in \cite{14, 23}. In this paper we will restrain to the homogeneous case.
\par 
The $CR$ structure of a homogeneous $CR$ manifold $M_0$ is efficiently described by 
considering  its 
\textit{$CR$ algebra} at any point
$p_0\in{M}_0$: it is the pair $(\kt_0,\vt)$ consisting 
of the real Lie algebra $\kt_0$ 
of its transitive group $\Kf_0$ of $CR$-automorphisms and of the subspace
$\vt=d\pi^{-1}(T^{0,1}_{p_0}M_0)$ of the complexification $\kt$ of $\kt_0$ (see \cite{MN05}).
The \textit{formal integrability} of the partial complex structure $T^{0,1}M_0$ 
of $M_0$ is  equivalent to
the fact that $\vt$ is a complex Lie subalgebra of~$\kt.$ The intersection $\vt\cap\bar{\vt}$
(conjugation is taken with respect to the real form $\kt_0$) is the complexification of the Lie algebra
of the stabilizer of $p_0$ in $\Kf_0$ and the quotient $\vt/(\vt\cap\bar{\vt})$ represents the
space $T^{0,1}_{p_0}M_0$ of \textit{anti-holomorphic}
complex tangent vectors at $p_0.$ \par 
We call $\nt$-reductive a homogeneous $CR$ manifolds for which 
\mbox{$\vt=(\vt\cap\bar{\vt})\oplus\nt(\vt),$} i.e. for which $T^{0,1}_{p_0}M_0$ can be identified to
the nilradical of~$\vt.$  It was shown in \cite{AMN2013} that the intersection
of any pair of
Matsuki-dual orbits in a complex flag manifold
$M,$ with the $CR$ structure inherited from $M,$ 
is an $\nt$-reductive
compact homogeneous $CR$ manifold. 
Moreover, when $M_0$ is $\nt$-reductive,
$\vt$ is the Lie algebra of a closed complex Lie subgroup $\Vf$ of $\Kf$ that contains the
stabilizer of $p_0$ as  its maximal compact subgroup, so that $M_0=\Kf_0/\Vf_0 \hookrightarrow
M_-=\Kf/\Vf$ is a generic $CR$-embedding. 
Vice versa, if $M_-$ is a $\Kf$-homogeneous complex
algebraic manifold, then
 a minimal $\Kf_0$-orbit $M_0$ in $M_-$ is an 
$\nt$-reductive compact homogeneous $CR$ manifold.
\par
Since $\Kf_0$ is a maximal compact subgroup of  a
linear algebraic complex group $\Kf,$ the quasi-projective 
manifold 
$M_-$ can be viewed as a $\Kf_0$-equivariant fiber bundle on the basis
$M_0$ (see \cite{Most62}).
We use this Mostow fibration of $M_-$ onto $M_0$ to construct a nonnegative 
\textit{smooth}
exhaustion
$\phiup$ of $M_-$, with $\phiup^{-1}(0)=M_0,$ to relate the Dolbeault cohomology of $M_-$ to
the cohomology of the tangential $CR$-complex on $M_0.$
 This requires some precision on the structure of the fibers and forces us to introduce a further
 requirement on the $CR$ algebra $(\kt_0,\vt),$ namely to ask that, if
 $\wt$ is the largest complex subalgebra of $\kt$ with \mbox{$\vt\subset\wt\subset(\vt+\bar{\vt}),$}
 (see \cite[Theorem 5.4]{MN05}), then 
 $\nt(\wt)$ is the nilradical of a parabolic subalgebra of $\kt.$ 
 This condition is satisfied 
 in many examples coming from Matsuki duality (cf. \cite{Mats88})
 and can always be satisfied by 
 \textit{strengthening} the $CR$ structure of 
 an $\nt$-reductive $M_0.$ \par
 When we drop this extra assumption, we are still able to construct a 
 \textit{continuous} exhaustion, which, when $M_0$ is $r$-pseudoconcave, 
is still strictly $r$-pseudoconcave, 
 allowing us to obtain results on the first $(r-1)$ 
 tangential Cauchy Riemann and Dolbeault
 cohomology groups of $M_0$ and $M_-$ 
 (or up to $(r-\mathrm{hd}(\mathcal{F})-1)$ 
 if we discuss cohomology with coefficients in a coherent
 sheaf $\mathcal{F}$).\par 
 Earlier versions of some results proved here  
were discussed in \cite{Marini,
MaNa1}.
 \par \vspace{3pt}
 The paper is organized as follows. \par 
 In \S\ref{sc4} we discuss some basic facts on 
 $\nt$-reductive $CR$ manifolds.
 We skip from basic stuff  on $CR$ manifolds and $CR$ algebras, for which we 
 refer, e.g.,  to \cite{40, MN05}, and only explain those special features which are necessary 
 for the developments of the next sections.
 \par 
Cartan and Mostow fibrations are related to the structure of  
negatively curved 
Riemannian symmetric
space 
of the set of Hermitian  symmetric matrices
with determinant one. Hence we 
found convenient to
discuss  in
\S\ref{sec3}, as a preliminary, some topics of the geometry of 
$\SL_n(\C)/\SU(n).$ 
\par 
In   \S\ref{hermitianfiber} we study decompositions of $\Kf$ 
with Hermitian fibers.  \par
Example~\ref{esampduesette}
shows that a $\Kf_0$-equivariant fibration of $M_-$ 
with Hermitian fibers,
as in \cite{Most55}, is not always possible.
In \S\ref{sec5} we  
describe the general structure of the fibers. To this aim,  
we consider a class of parabolic subalgebras associated to the pair
$(\kt_0,\vt)$ 
 and find a condition, 
that we call {\HNR} from \textit{horocyclic
nilradical}, under which we get
a Mostow fibration of $M_-$ with Hermitian fibers.
 \par
In the final section \S\ref{cohomo} we apply these results  to construct
an exhaustion function which permits to relate some cohomology groups of
the tangential 
$CR$ complexes on $M_0$ 
to the corresponding cohomology groups of the
Dolbeault complexes on $M_-$ and analogous results for \v{C}ech cohomology with
coefficients in a coherent sheaf.
We conclude with 
the study of an example of a family of intersections of Matsuki-dual orbits
and an application of \S\ref{hermitianfiber} to obtain a pseudoconcavity result for 
which we do not require the validity of the {\HNR} assumption.
\par

\section{Compact homogeneous \texorpdfstring{$CR$}{TEXT} 
manifolds and 
\texorpdfstring{$\mathfrak{n}$}{TEXT}-reductiveness} \label{sc4}
In this section we introduce the class of homogeneous $CR$ manifold which is
the object of this investigation. We found convenient to recall, in an
initial  short subsection,
the definition of reductive Lie group, as it is not completely standard  in the literature.
\subsection{Reductive Lie groups}
We call \emph{reductive} a Lie algebra $\kt$ whose radical is abelian:
its commutator subalgebra 
$[\kt,\kt]$ is its semisimple ideal and its radical $\at$ equals its center (see \cite{bour1}).\par
Reductive $\kt$'s are characterized by having faithful semisimple representations.\par
An involution $\thetaup$ on a Lie algebra $\kt$ yields a direct sum decomposition
\begin{equation*} \kt=\kt_0\oplus\po,\;\;\text{with}\;\;
 \kt_0=\{X\in\kt\mid \thetaup(X)=X\},\;\; 
  \po=\{X\in\kt\mid \thetaup(X)=-X\}.
\end{equation*}
\par
A  Lie group $\Kf$ is \emph{reductive} (see \cite{Kn:2002})
if its Lie algebra $\kt$ is reductive and, moreover,
there are  an involution $\thetaup$ and an invariant bilinear form $\frm$ on 
$\kt$ such that 
\begin{enumerate}
  \item[$(i)$] $\kt_0\perp\po$ for $\frm$;
\item[$(ii)$] $\frm<0$ on $\kt_0$ and $\frm>0$ on $\po$;
 \item[$(iii)$] $\kt_0$ is the Lie algebra of a compact subgroup $\Kf_0$ of $\Kf$ and
 \begin{equation}\label{r3.1}
 \Kf_0\times\po\ni (x,X)\longrightarrow x\cdot\exp(X)\in\Kf\end{equation}
 is a diffeomorphism onto;
 \item[$(iv)$] every automorphism $\Ad(x)$ of the complexification $\kt^{\C}$ of $\kt$, with $x\in\Kf$,
 is inner, i.e. belongs to the analytic subgroup of the automorphis group of $\kt^{\C}$ having
 Lie algebra $\ad(\kt)$.
\end{enumerate}
Then: $\thetaup$ is a \emph{Cartan involution}, 
$\kt=\kt_0\oplus\po$ and \eqref{r3.1} are
\emph{Cartan decompositions}, $\Kf_0$ is the associated \emph{maximal compact subgroup},
$\frm$ is the \emph{invariant bilinear form}.
The maximal compact subgroup $\Kf_0$ of $\Kf$ 
intersects all connected component of $\Kf$ (see \cite[Proposition 7.19]{Kn:2002}).
In particular, $\Kf$ has finitely many connected components.
\subsection{Splittable Lie subalgebras 
}

Let ${\kt}$ be a reductive complex Lie algebra, and 
\begin{equation*}
  {\kt}=\mathfrak{z}\oplus\mathfrak{s},\quad\text{with}\quad
\mathfrak{z}=\{X\in\kt\mid [X,{\kt}]=\{0\}\},\;\;
\mathfrak{s}=[{\kt},{\kt}]
\end{equation*}
its decomposition into the direct sum of its center and its semisimple
ideal. 
An element $X$ of $\kt$ is 
\emph{semisimple} if
$\mathrm{ad}(X)$ is a semisimple derivation of ${\kt}$, 
and \emph{nilpotent} if $X\in\mathfrak{s}$ and $\mathrm{ad}(X)$ is
nilpotent.\par 
An equivalent formulation is obtained by
considering a faithful matrix representation of $\kt$ in
which the elements of $\mathfrak{z}$ are diagonal: then semisimple and nilpotent
elements correspond to semisimple and nilpotent matrices, respectively.
\par
Each $X\in{\kt}$ admits a unique Jordan-Chevalley decomposition
\begin{equation*}
  X=X_s+X_n,\quad \text{with $X_s$ semisimple, $X_n$ nilpotent, and
$[X_s,X_n]=0$.}
\end{equation*}
A Lie subalgebra $\mathfrak{v}$ of $\kt$ is \emph{splittable} if, 
for each $X\in\mathfrak{v}$,
both $X_s$ and $X_n$ belong to~$\mathfrak{v}$. \par
If $\mathfrak{v}$ is a Lie subalgebra of ${\kt}$, 
the set
\begin{equation*}
\nr(\mathfrak{v})=\{ X\in\rad(\mathfrak{v})\mid X\; \mathrm{is\; nilpotent} \}
\end{equation*}
is 
a nilpotent ideal  of $\mathfrak{v}$, 
with
\begin{equation*}
  \radn(\mathfrak{v})=\rad(\mathfrak{v})\cap [\mathfrak{v},\mathfrak{v}]
\subset \nr(\mathfrak{v})\subset\nil(\mathfrak{v}),
\end{equation*}
where $\nil(\mathfrak{v})$ is the nilradical, i.e. the maximal nilpotent ideal
of $\mathfrak{v}$, and $\radn(\mathfrak{v})$ its nilpotent radical, i.e. the 
intersection of the kernels of all irreducible finite dimensional
linear representations of
$\mathfrak{v}$. Note that the nilpotent ideal $\nr(\mathfrak{v})$, unlike
$\nil(\mathfrak{v})$ and $\radn(\mathfrak{v})$, depends on the inclusion
$\mathfrak{v}\subset{\kt}$ (cf. \cite[\S 5.3]{Bou75}). 
We recall 
\begin{prop}[{see \cite[\S 5.4]{Bou75}}]
 Every splittable Lie subalgebra $\mathfrak{v}$ 
admits a Levi-Chevalley decomposition 
\begin{equation}\label{eq:42}
 \mathfrak{v}=\nr(\mathfrak{v})\oplus\vt_r,
\end{equation}
with $\vt_r$ reductive  and uniquely determined modulo
conjugation by elementary automorphisms of $\mathfrak{v}$, i.e. finite products of
automorphisms of the form $\exp(\ad(X))$, with $X\in\mathfrak{v}$ and nilpotent.\qed
\end{prop}
\subsection{Definition of $\mathfrak{n}$-reductive}\label{subsreductive}
Let ${\kt}$ be the complexification of 
a compact Lie algebra ${\kt}_0$. 
Conjugation in ${\kt}$ 
will be 
understood
with respect to 
its compact
real form ${\kt}_0$. Note that all 
Lie 
subalgebras of
a compact Lie algebra are compact and hence reductive.
\begin{prop} \label{prop:31}
For any complex Lie subalgebra
$\mathfrak{v}$ of ${\kt}$, the intersection  
$\mathfrak{v}\cap\bar{\mathfrak{v}}$ is reductive and splittable. 
In particular,
$  \mathfrak{v}\cap\bar{\mathfrak{v}}\cap\nr(\mathfrak{v})=\{0\}$. A
splittable $\mathfrak{v}$ 
admits a Levi-Chevalley decomposition with a reductive Levi factor containing
  $\mathfrak{v}\cap\bar{\mathfrak{v}}$.
\end{prop}
\begin{proof} We recall that $\mathfrak{v}$ is splittable if and only if 
its radical 
is splittable (\cite[Ch.VII, \S{5}, Th\'eor\`eme 2]{Bou75}). In this case, 
$\mathfrak{v}$ admits a Levi-Chevalley decomposition and all maximal reductive
Lie subalgebras of $\mathfrak{v}$ can be taken as 
reductive Levi factors.
The intersection $\mathfrak{v}\cap\bar{\mathfrak{v}}$ is reductive, being
 the complexification of the compact Lie algebra
$\mathfrak{v}\cap{{\kt}}_0$. Then the reductive Levi factor in the
Levi-Chevalley decomposition of $\mathfrak{v}$ can be taken to contain
$\mathfrak{v}\cap\bar{\mathfrak{v}}$ (see e.g. \cite{OV90}).\end{proof}

\begin{ntz} In the following, for a complex 
Lie 
subalgebra $\mathfrak{v}$ of $\kt$,
we shall use the notation
\begin{equation*}
  \Li_0(\mathfrak{v})=\mathfrak{v}\cap\kt_0,\quad \Li(\mathfrak{v})=
\mathfrak{v}\cap\bar{\mathfrak{v}}.
\end{equation*}
\end{ntz}

\begin{defn}\label{df33}
Let $\mathbf{K}_0$ be a compact Lie group with Lie algebra
$\kt_0$ and $M_0$ a $\mathbf{K}_0$-homogeneous $CR$ manifold,
with isotropy $\If_0$ and $CR$ algebra
$({\kt}_0,\mathfrak{v})$ at a point
$p_0\in{M}_0$.
We say that $M_0,$ and 
its 
$CR$ algebra $({\kt}_0,\mathfrak{v})$, 
are 
\emph{$\mathfrak{n}$-reductive} if 
\begin{displaymath}
\mathfrak{v}=\nr(\mathfrak{v})\oplus\Li(\mathfrak{v}),
\end{displaymath}
i.e. if 
$\Li(\mathfrak{v})=\mathfrak{v}\cap\bar{\mathfrak{v}}$
is a reductive complement of
$\nr(\mathfrak{v})$ in $\mathfrak{v}$.  
\end{defn}
\begin{rmk}
If $(\kt_0,\mathfrak{v})$ is $\mathfrak{n}$-reductive, then
$\mathfrak{v}$ is splittable. Indeed all elements of 
$\nr(\mathfrak{v})$ are nilpotent and all elements of 
$\Li(\mathfrak{v})$ are splittable, because $\Li(\mathfrak{v})$
is the complexification of $\Li_0(\mathfrak{v})$, 
which is
splittable because consists of semisimple elements.  
Then $\vt$ is splittable by \cite[Ch,VII, \S{5}, Th\'eor\`eme 1]{Bou75}.
\end{rmk} 
All submanifolds which are intersections of dual submanifold in the Matsuki duality,
with the $CR$ structure inherited by the embedding in the ambient flag manifold,
are $\nt$-reductive (see \cite[\S{1}]{AMN2013}).
We exhibit here
an example of a compact homogeneous $CR$ manifold  $M_0$ which is not 
$\mathfrak{n}$-reductive.
\begin{exam}
Let $\mathbf{K}_0=\mathbf{SU}(n)$, $n\geq 3$. Fix a complex symmetric 
nondegenerate $n\times{n}$ matrix $S$ 
and consider the subgroup $\mathbf{V}=\{a\in\mathbf{SL}(n,\mathbb{C})\mid
a^tSa=S\}$ of $\mathbf{SL}(n,\mathbb{C})$, with Lie algebra
$\mathfrak{v}=\{X\in\mathfrak{sl}(n,\mathbb{C})\mid
X^tS+SX=0\}$. 
Set $\If_0=\mathbf{V}\cap\mathbf{K}_0$ and
$M_0=\mathbf{K}_0/\If_0$. This is a $\mathbf{K}_0$-homogeneous 
$CR$ manifold with $CR$ algebra 
$(\kt_0,\mathfrak{v})$, where
$\kt_0\simeq\mathfrak{su}(n)$, $\mathfrak{v}\simeq\mathfrak{so}(n,\mathbb{C})$.
If $S$ and $S^*$ 
are linearly independent, then $\mathfrak{v}$
is a semisimple Lie subalgebra of $\kt$ distinct from
$\mathfrak{v}\cap\bar{\mathfrak{v}}$.
\end{exam}
The $CR$ manifolds of Definition~\ref{df33} have canonical complex realizations:
\begin{thm}[{\cite[Theorem {4.3}]{AMN2013}}]
 Let $M_0$ be an $\nt$-reductive $\Kf_0$-homogeneous $CR$ manifold, with 
 $CR$ algebra $(\kt_0,\vt)$ and isotropy $\If_0$ at some point $p_0\in{M}_0.$
 Then there is a closed complex Lie subgroup $\Vf$ of the complexification $\Kf$ of
 $\Kf_0$ with $\Kf_0\cap\Vf=\If_0$ and $\Lie(\Vf)=\vt$ such that
 the canonical map 
\begin{equation}
 M_0\simeq\Kf_0/\If_0\longrightarrow M_-=\Kf/\Vf
\end{equation}
is a generic $CR$ embedding. \qed
\end{thm}
\begin{rmk} Vice versa, if $M_-=\Kf/\Vf$ is the homogeneous complex manifold of 
the complexification $\Kf$ of $\Kf_0$, it is shown in  \cite[Prop.2.9]{AMN2013} that
any $\Kf_0$-orbit $M_0$ 
of minimal dimension in $M_-$, with the $CR$ structure induced by 
the ambient space, is  $\nt$-reductive. 
\end{rmk}
 \section{Some remarks on \texorpdfstring{$\SL_n(\C)/\SU(n)$}{TEXT}} \label{sec3}
 Keep the notation of \S\ref{sc4}. As we explained in the introduction, 
 we need to precise the structure of the fibers of the 
 $\Kf_0$-equivariant 
 Mostow fibration $M_-\to{M}_0.$  
 \par 
 Mostow fibration (\cite{Most55, Most62}) extends to homogeneous spaces the 
Cartan decomposition of reductive Lie groups. Both are related
to the fact that the positive definite $n\times{n}$ Hermitian symmetric matrices 
with determinant one are the points of a
Riemannian symmetric
space $\Mf_n$ with negative sectional curvature. We will discuss some topics
on the geometry of $\Mf_n$ (see e.g. \cite{Eb}). 
\par \smallskip
Any compact Lie group $\Kf_0$ has,
for some integer $n>1$,  a faithful linear representation in $\SU(n),$
which extends to a linear representation $\Kf\hookrightarrow\SL_n(\C)$.
Thus decompositions in $\SL_n(\C)$ are preliminary to the general case.
\par\smallskip

The linear group $\SL_n(\C)$ has the Cartan decomposition 
\begin{equation*}
 \SU(n)\times\po(n)\ni (x,X) \longrightarrow x\cdot\exp(X)\in\SL_n(\C),
\end{equation*}
where $\SU(n)=\{x\in\SL_n(\C)\mid x^*x=\Id_n\}$ 
is its maximal compact subgroup consisting
of $n\times{n}$ unitary matrices with determinant one, 
and $\po(n)$ the subspace  of the
traceless Hermitian symmetric $n\times{n}$
matrices in $\slt_n(\C)$.\par 
 The quotient $\Mf_n=\SL_n(\C)/\SU(n)$ 
is a symmetric space of
the noncompact type
and rank $(n\! - \! 1)$, endowed with a Riemannianmetric with negative curvature. 
We can identify $\Mf_n$ with the set $\Pfo(n)$
of positive definite Hermitian symmetric matrices 
in $\SL_n(\C),$
which in turn
is diffeomorphic to 
$\po(n)$ via the exponential map. In this way $\Mf_n$ can be considered as 
an open subset of
$\po(n)$ and 
its tangent bundle 
${T}\Mf_n$ is naturally diffeomorphic to the subbundle 
\begin{equation*}
 T\Mf_n=\{(p,X)\in\Mf_n\times\pt(n)\mid p^{-1}X\in\po(n)\}
\end{equation*}
of the trivial bundle
$\Mf_n\times\pt(n),$ where we set $\pt(n)=\{X\in\C^{n\times{n}}\mid X^*=X\}.$
\par
The special linear group 
$\SL_n(\C)$ acts on $\Mf_n$ as a group of isometries,  
by 
\begin{equation*}
 \SL_n(\C)\times\Mf_n\ni (z,p)\longrightarrow z p z^* \in \Mf_n,
\end{equation*}
and $\SU(n)$ is the stabilizer of the identity
$e=\Id_n$, that we choose as the base point.\par 
The metric tensor on $\Mf_n$ is 
\begin{equation*} 
(X,Y)_p=
 g_p(X,Y)=\trac(p^{-1}Xp^{-1}Y),\;\;\forall p\in\Mf_n,\;\;\forall X,Y\in{T}_p\Mf_n.
\end{equation*}
\par 
The curves 
\begin{equation*}
 \R\ni{t}\to z\exp(tX) z^*\in\Ml_n,\quad\textit{ for }X\in\po(n),\,z\in\SL_n(\C)
\end{equation*}
are the complete geodesics in $\Mf_n$ issued from $p=z z^*$ and
 \begin{equation*}
 \dist(p_1,p_2)\! =\left({{\sum}_{i=1}^n|\log(\lambdaup_i(p_1^{-1}p_2)|^2}\right)^{1/2},
\end{equation*} 
 where $\lambdaup_i(p_1^{-1}p_2)$ are the eigenvalues of the matrix $p^{-1}p_2,$ which are
 real and positive, the Riemanniandistance on $\Mf_n.$
 

\subsection{Killing and Jacobi vector fields}\label{kiljac} 
Since $\Mf_n$ is a 
Riemanniansymmetric space of
$\SL_n(\C)$, the Lie algebra of its Killing vector fields is isomorphic to
$\slt_n(\C).$ The correspondence is
\begin{equation*}
 \slt_n(\C)\ni{Z}\longrightarrow \zetaup_Z=\{p\to 
 Z p+p Z^*\}\in\Xg(\Mf_n).
\end{equation*}
\par
\smallskip
For $H$ in $\po(n),$ the restriction to $[0,1]$ of the geodesic
 $t\to\gammaup_H(t)=\exp(t H)$ is the shortest path
 from $e=\gammaup_H(0)$
to $h=\exp(H)=\gammaup_H(1).$ We will denote by $\Jf(H)$ the space of Jacobi vector fields
on $\gammaup_H$ and by $\Jf_0(H)$ its subspace consisting of those vanishing at $t=0.$
For each $Z\in\slt_n(\C)$, 
the restriction  of $\zetaup_{Z^*}$ to $\gammaup_H$ is a Jacobi vector field,
that we  denote by $\thetaup_Z$:
\begin{equation*}
\{\R\ni{t}\longrightarrow \thetaup_Z(t)=Z^*\exp(tH)+\exp(tH)Z\}\in\Jf(H).
\end{equation*}
To describe  
$\Jf(H)$ 
it is convenient
to consider the commutator of $H:$ 
\begin{equation*} \left\{
\begin{aligned}&\Ct(H)=\{Z\in\slt_n(\C)\mid [Z,H]=0\}=\Ct_u(H)\oplus\Ct_0(H),\;\;\text{with} \\
&\Ct_u(H)=\Ct(H)\cap\su(n),\;\; \Ct_0(H)=\Ct(H)\cap\po(n).\end{aligned}\right.
\end{equation*}
\begin{prop} \label{Prop2.1}
The  
correspondence $\thetaup:\slt_n(\C)\ni{Z}\to\thetaup_Z\in\Jf(H)$ 
is a linear map with 
kernel $\Ct_u(H).$ For each $T\in\Ct_0(H)$, $J(t)=t\cdot\thetaup_T(t)$ is a Jacobi vector field
and 
\begin{align}\label{j3.24an}
 \Jf(H)&=\{\thetaup_Z+t\!\cdot\!\thetaup_T\mid Z\in\slt_n(\C),\;\; T\in\Ct_0(H)\},\\
 \label{j3.25n}
 \Jf_0(H)&=\{\thetaup_Y+t\!\cdot\!\thetaup_{T}\!\mid Y\in\su(n),\;\; T\in\Ct_0(H)\}.
\end{align}
Fix $Z\in\slt_n(\C)$ and $T\in\Ct_0(H).$ Then 
\begin{equation}\label{defjei}
J(t)=\thetaup_Z(t)+t\!\cdot\!\thetaup_T(t)=Z^*\exp(tH)+\exp(tH)Z+2t\cdot{T}\cdot\exp(tH)
 \end{equation}
is the Jacobi vector field on $\gammaup_H$ satisfying the initial conditions:
\begin{equation} \label{j3.27n}
\begin{cases}
 J(0)=Z+Z^*,\\[4pt]
 \dot{J}(0)=\tfrac{1}{2}[H,Z-Z^*]+2T,
\end{cases} \end{equation}
and we have 
\begin{equation} \label{j3.27na}
\begin{cases}
 \dot{J}(t)=\tfrac{1}{2}\thetaup_{[H,Z]+2T}(t),\\[4pt]
 \dfrac{D^k J(t)}{dt^k}=2^{-k}\thetaup_{\ad_H^k(Z)}(t), &\text{for $k\geq{2}.$}
\end{cases}
\end{equation}
 \end{prop} 
\begin{proof} 
If $T\in\Ct_0(H),$ then $\thetaup_T$ is  parallel  and therefore
also $t\!\cdot\!\thetaup_T$ is   Jacobi on $\gammaup_H.$
 To compute 
 the covariant derivatives of the Jacobi vector field
 $J(t)$ defined in \eqref{defjei},
we use the parallel transport $T_{\gammaup_H(t)}\Ml_n\ni{X}\to \exp(sH/2)X\exp(sH/2)
 \in{T}_{\gammaup_H(t+s)}\Ml_n$ along $\gammaup_H$. Then 
\begin{align*}
 \dot{\thetaup}_Z(t)&=\left(\frac{d}{ds}\right)_{s=0} \left[\exp(-sH/2)\left\{Z^*\exp([t+s]H)+
 \exp([t+s]H)Z\right\} \exp(-sH/2)\right]\\
 &=\tfrac{1}{2}[Z^*,H]\exp(tH)+\tfrac{1}{2}\exp(tH)\,[H,Z]=\tfrac{1}{2}\thetaup_{[H,Z]}(t).
 \end{align*}
By iteration we obtain 
 \eqref{j3.27na} and, in particular, \eqref{j3.27n}. \par
 Finally, 
 we need to show that all $J$ in $\Jf(H)$ have the form \eqref{defjei}.
 Since $\ad_H$ is semisimple,  $\slt_n(\C)$
 decomposes into the direct sum of its image and its kernel. Hence
 $\po(n)=[H,\su(n)]\oplus\Ct_0(H),$ and this yields 
 \eqref{j3.24an} and 
 \eqref{j3.25n}.
\end{proof}
For $X\in\po(n),$ we will denote by $J_X$ the geodesic on $\gammaup_H$ with 
\begin{equation} \label{jeix}
\begin{cases}
 J_X(0)=0,\\
 \dot{J}_X(0)=X,
\end{cases}
 \end{equation}
 while $\thetaup_X\in\Jf(H)$ satisfies $\thetaup_X(0)=2X,$ $\dot{\thetaup}_X(0)=0.$
\par
The nonconstant geodesics of a manifold with negative curvature have no conjugate points.
Hence the map
$\Jf_0(H)\ni{J}\to{J}(t)\in{T}_{\gammaup_H(t)}\Mf_n$ is a linear isomorphism
for all $t\neq{0}$. Moreover, for every $J\in\Jf(H)$, the real
map\footnote{
Here and in the following we drop the subscript  indicating where norms and scalar products are
computed, when we feel that this simplified notation does not lead to ambiguity.}  
$t\to\|J(t)\|$ is 
nonnegative and
convex  
and
therefore a nonzero $J(t)\in\Jf(H)$ vanishes for at most one value of $t\in\R,$
corresponding to a minimum of $\|J(t)\|^2$ and thus to a solution of
$(J(t)|\dot{J}(t))=0.$ 
\begin{lem} \label{lem:3.13.3} If $J\in\Jf(H)$ is not parallel along $\gammaup_H$ and
 $(J(0)|\dot{J}(0))=0$, then 
 \begin{equation*}\vspace{-22pt}
 \|J(0)\|<\|J(t)\|\,\,\text{for all $t\neq{0}$.} 
\end{equation*}
\qed

\end{lem} 
\begin{lem}
The quadratic form
\begin{equation}\label{J18n} 
 \| J\|_H^2 = \int_0^1\!(1\! -\!{t})\left( \|\dot{J}(t)\|^2
  +
  (J(t),\ddot{J}(t))
  \right) dt 
\end{equation}
 is positive semidefinite on $\Jf(H)$ and 
\begin{equation*}
 \| J\|^2_H=0\Leftrightarrow J=
 \thetaup_{T}\;\text{   for a }\; T\in\Ct_0(H).\end{equation*}
\end{lem}
\begin{proof}
Let $J\in\Jf(H).$  Then 
$
 (\ddot{J},J)=-(R(J,\dot{\gammaup}_H)\dot{\gammaup}_H|J)\geq{0}
$ for all $t$ by the Jacobi equation, 
because $\Ml_n$ has negative sectional curvature. 
Hence $\|J\|^2_H=0$ if and only if 
$\dot{J}(t)=0$ for all $t$.
The statement follows because $\{\thetaup_T\mid T\in\Ct_0(H)\}$ is the space of
the Jacobi vector fields 
that are parallel along $\gammaup_H$.
\end{proof}
\begin{lem}\label{lemJ6}
We have 
\begin{align}\label{intparti}
  \|J(1)\|^2
  =\|J(0)\|^2\! +\!  2
  (J(0)|\dot{J}(0))
 \! +2\, \|J\|_H^2,\qquad 
  \forall J\in\Jf(H). 
\end{align}
\end{lem} 
\begin{proof} We 
apply the integral form of the reminder in the first order Taylor's expansion to 
 $f(t)=\|J(t)\|^2.$
\end{proof}
For further reference, we state an easy consequence of Lemma~\ref{lemJ6}.
\begin{lem}\label{lemma3.2.13}
Let $Z\in\slt_n(\C)$, $X\in\pt_0(n)$, and $\trac(X\cdot{Z})=0.$
Then 
\begin{equation}
\label{eq:3.35}
 \|\thetaup_Z(1)-J_X(1)\|^2=\|Z+Z^*\|^2+ 2(H|[Z,Z^*])
 +
 \|\thetaup_Z-J_X\|_H^2.
\end{equation}
\end{lem} 
\begin{proof} We apply \eqref{intparti} 
to
$J=\thetaup_Z-J_X.$ 
\par 
Then 
\begin{equation*}
 J(0)=Z+Z^*,\quad \dot{J}(0)=\tfrac{1}{2}[H,Z-Z^*] - X.
\end{equation*}
yields 
\begin{align*}
  \|\thetaup_Z(1)-J_X(1)\|^2&=\| J(1)\|^2  = \|Z+Z^*\|^2+2(Z+Z^* |X
 +\tfrac{1}{2}[H,Z-Z^*])+(J|J)_H\\
 &=\|Z+Z^*\|^2+ ([H,Z-Z^*]|Z+Z^*)+(J|J)_H\\ 
 &= \|Z+Z^*\|^2+2  (H |[Z,Z^*])+ (J|J)_H.
\end{align*}
\end{proof}
\par\smallskip
Let $J(t)=\thetaup_Z(t)+t\thetaup_T(t),$ 
with $Z\in\slt_n(\C)$ and $T\in\Ct_0(H).$ The two commuting Hermitian symmetric matrices
$H$ and $T$  can be simultaneously diagonalized in an orthonormal basis of $\C^n.$ 
Let $\lambdaup_1,\hdots,\lambdaup_m$ be the distinct eigenvalues of $H$, with multiplicities
$n_1,\hdots,n_m$ and choose an orthonormal basis of $\C^n$ to get  matrix representations 
\begin{equation}\label{matric}\left\{
\begin{aligned}
 H&= 
{\begin{pmatrix}
 \lambdaup_1\Id_{n_1}\\
 &\lambdaup_2\Id_{n_2}\\
 && \ddots\\
 && &\lambdaup_m\Id_{n_m}
\end{pmatrix}}, 
&T&= 
{\begin{pmatrix}
 \tauup_1\\
 & \tauup_2\\
 &&\ddots\\
 &&& \tauup_m
\end{pmatrix}},\\ 
Z&=
\begin{pmatrix}
 z_{1,1} & z_{1,2}&\hdots&z_{1,m}\\
 z_{2,1} & z_{2,2} &\hdots & z_{2,m}\\
 \vdots&\vdots &\ddots&\vdots \\
 z_{m,1}&z_{m,2}&\hdots & z_{m,m}\end{pmatrix}, 
 && 
\begin{aligned}
& \text{with $\tauup_i\in\R^{n_i\times{n}_i}$ diagonal,}\\
& \text{and $z_{i,j}\in\C^{n_i\times{n}_j}.$}
\end{aligned}
\end{aligned}\right.\end{equation}
Let us extend the trace norm of $\po(n)$ to a norm 
in $\slt_n(\C),$ by setting 
\begin{equation*}
 |\| A\|| =\sqrt{\trac(AA^*)}\geq{0},\;\;\forall A\in\slt_n(\C).
\end{equation*}
Then
\begin{align*}
& \| J(t)\|^2= \trac\left(Z^2+{Z^*}^2
\!
+
\!
 2e^{tH}Z e^{-tH}Z^*
 \!
 +
 \!
 4t(Z+Z^*)T
 \!
 +
 \!
 4t^2T^2
 \right)\\
&\; =\trac\left(2\re{\sum}_{i,j=1}^m{z_{i,j}z_{j,i}}\!
+\!
2{\sum}_{i,j=1}^m{z}_{i,j}z_{i,j}^* e^{t(\lambdaup_i-\lambdaup_j)}
 \!+
 \!
 8t\re{\sum}_{i=1}^m\tauup_i z_{i.i}\!
 +\!
 4t^2\!{\sum}_{i=1}^m\tauup_i^2\right)\\
 &\; = {\sum}_{i\neq{j}}\left|\left\|\,  z_{i,j} e^{t(\lambdaup_i-\lambdaup_j)/2}+{z}^*_{j,i} 
 e^{t(\lambdaup_j-\lambdaup_i)/2}\,
 \right\|\right|^2+{\sum}_{i=1}^m |\|\,2t \tauup_i+ z_{i,i}+\bar{z}_{i,i}\, |\|^2.
\end{align*}
\par 
Set $Z(t)=\exp(t H/2) Z \exp (- t H/2)=(z_{i,j}(t)),$ with 
\mbox{$z_{i,j}(t)\! =\! z_{i,j}e^{t(\lambdaup_i-\lambdaup_j)/2}
\in\C^{n_i\times{n}_j}.$}
We obtain the expression
\begin{equation}\label{jeit}
 \|J(t)\|^2={\sum}_{i\neq{j}}|\| z_{i,j}(t)\! +\!
  {z}^*_{j,i}(t)|\|^2
  \!
  +
  \!
  {\sum}_{i=1}^m|\|2t\tauup_i+z_{i,i}
  \!
  +
  \!
  {z}^*_{i,i}|\|^2.
\end{equation} 
If $J(t)=0$, then each summand in \eqref{jeit} equals zero. For the terms 
in the first sum this amounts to the fact that
$[H,Z(t)]=((\lambdaup_i-\lambdaup_j)z_{i,j}(t))_{1\leq{i,j}\leq{m}}$
is Hermitian symmetric. Since $[H,Z(t)]$ and $[H,Z]$ are similar, we obtain:
\begin{lem} Let $Z\in\slt_n(\C)$ and $H\in\po(n).$
 A necessary condition in order that there exists $T\in\Ct_0(H)$ such that the Jacobi vector field
 $J(t)=\thetaup_Z(t)+t\thetaup_T(t)$ on $\gammaup_H$ 
 vanishes at some $t\in\R$ is that $[H,Z]$ is semisimple with
 real eigenvalues. \qed
\end{lem}

\begin{exam}\label{esampduesette}
We consider the matrices \begin{equation*}
H= \begin{pmatrix}
\lambdaup_1 & 0 & 0\\
0 & \lambdaup_2 & 0 \\
0 & 0 &\lambdaup_3
\end{pmatrix}\in\slt_3(\R),\;\; 
Z= \begin{pmatrix}
0 & a & 0 \\
b & 0 & c \\
0 & d & 0
\end{pmatrix},\;\; Y= \begin{pmatrix}
0 &\alphaup & 0 \\
-\bar{\alphaup} &0 &  {\betaup}\\
0 & -\bar{\betaup} & 0
\end{pmatrix}.
\end{equation*}
We impose the conditions that 
$Z$  be nilpotent and orthogonal to
$X=[H,Y]$ and that $\thetaup_{Z+Y}(1)=0.$ 
This translates into the set of equations \begin{equation*}
\begin{cases} ab+cd=0,\\
(\lambdaup_2-\lambdaup_1)(a\bar{\alphaup}+b\alphaup)+(\lambdaup_3-
\lambdaup_2)(c\bar{\betaup}+d{\betaup})=0,\\
\alphaup=(ae^{\lambdaup_1}+\bar{b}e^{\lambdaup_2})/(e^{\lambdaup_2}
-e^{\lambdaup_1}),\\
\betaup=(c e^{\lambdaup_2}+\bar{d} e^{\lambdaup_3})/(e^{\lambdaup_3}
- e^{\lambdaup_2}).
\end{cases}
\end{equation*}
By using the last two equation we reduce to the system \begin{equation*}
\begin{cases} ab+cd=0,\\[9pt]
\begin{aligned}
\dfrac{\lambdaup_2-\lambdaup_1}{e^{\lambdaup_2}-e^{\lambdaup_1}}
\big(|a|^2 e^{\lambdaup_1}+ab(e^{\lambdaup_1}+e^{\lambdaup_2})+
|b|^2 e^{\lambdaup_2}\big)\qquad\qquad\qquad \\
+
\dfrac{\lambdaup_3-\lambdaup_2}{e^{\lambdaup_3}-e^{\lambdaup_2}}
\big(|c|^2 e^{\lambdaup_2}+cd(e^{\lambdaup_2}+e^{\lambdaup_3})+
|d|^2 e^{\lambdaup_3}\big)=0\end{aligned}
\end{cases} 
\end{equation*}
Assuming $ab\neq{0}$ we obtain from the first equation $d=-ab/c$ 
and, as   ${\lambdaup_3=-\!\lambdaup_1\! -\! \lambdaup_2,}$
the system  reduces to \begin{equation}\tag{$*$}
\label{stellinaa} \left\{
\begin{aligned}
\dfrac{\lambdaup_2-\lambdaup_1}{e^{\lambdaup_2}-e^{\lambdaup_1}}
\left(|a|^2 e^{\lambdaup_1}+ab(e^{\lambdaup_1}+e^{\lambdaup_2})+
|b|^2 e^{\lambdaup_2}\right)\qquad\qquad\qquad \qquad\qquad \\
+\, 
\dfrac{\lambdaup_1+2\lambdaup_2}{e^{\lambdaup_2}-e^{-\lambdaup_1-\lambdaup_2}}
\left(|c|^2 e^{\lambdaup_2}-ab(e^{\lambdaup_2}+e^{-\lambdaup_1-\lambdaup_2})+
\frac{|ab|^2}{c^2} e^{-\lambdaup_1-\lambdaup_2}\right)=0\end{aligned}\right.
\end{equation}
Let us restrict to the case where $a,b,c$ are real.
For any fixed $a,b,c$ with $ab\neq{0},$  the left hand side of \eqref{stellinaa} is
positive when $ab>0$ and 
$|\lambdaup_1+2\lambdaup_2|$ is 
sufficiently small. Let us keep now $\lambdaup_1$ fixed and consider the left hand side of
\eqref{stellinaa} as a real valued
function $f(\lambdaup_2)$ of the parameter $\lambdaup_2$.
Then 
\begin{equation*}
 \lim_{\lambdaup_2\to +\infty} \lambdaup_2^{-1} f(\lambdaup_2)=|b|^2+|c|^2-ab.
\end{equation*}
If $ab>0,$ this is negative for $|a|\gg{1}.$  
Then we can choose the parameters to satisfy \eqref{stellinaa}.
In conclusion: \textsl{we can find  $H,Z,Y$
with $H\in\po(3),$ $Z\in\slt_3(\C)$ nilpotent, and $Y\in\su(3)$
with $X=[H,Y]\in\po(3)$ trace-orthogonal to $Z$ such that
$\thetaup_{Z+Y}(0)\neq{0}$ and $\thetaup_{Z+Y}(1)=0.$}
\end{exam}

\par\smallskip
Jacobi vector fields are used to compute the differential of the exponential map.
In fact, for $H,X\in\po(n)$, the covariant
derivative \mbox{$\tfrac{D}{dt}\exp(H+tX)|_{t=0}$} is 
the value at $t=1$ of the Jacobian vector field $J_X\in\Jf_0(H).$
If $X=[H,Y]+T,$ with $Y\in\su(n)$ and $T\in\Ct_0(H),$ then 
\begin{equation}
  \frac{D}{dt}\exp(H+tX)|_{t=0}=J_X(1)=[\exp(H),Y]+T\exp(H).
\end{equation}
\section{Decompositions with Hermitian fibers}\label{hermitianfiber}
 \subsection{Decomposition of \texorpdfstring{$\SL_n(\C)$}{TEXT}}
 Throughout this section, 
$\Vf$ is 
a closed complex Lie subgroup of $\SL_n(\C)$, that 
admits a Levi-Chevalley
decomposition $\Vf=\Vf_{\!{r}}\cdot\Vf_{\!{n}}$, with $\Vf_{\!{r}}$ 
algebraic reductive and $\Vf_{\!{n}}$ unipotent (cf. \cite[Ch.I, \S{6.5}]{OV93}).
We choose the embedding $\Vf\hookrightarrow\SL_n(\C)$ 
in such a way that 
$\Vf_0=\Vf\cap\SU(n)$ is a maximal compact sugbroup of $\Vf$ 
and a real form of
$\Vf_{\!{r}}$ 
and set:
\begin{gather}
\vt=\Lie(\Vf),\;\; \vt_r=\Lie(\Vf_{\!{r}}),\;\; \vt_n=\Lie(\Vf_{\!{n}}),\;\;
\vt_0=(\vt\cap\su(n))=\Lie(\Vf_0),\;\; \\ 
\label{emmezero}
\mt_0=(\vt+\vt^*)^\perp\cap\po(n), \quad
\vt=\vt_0\oplus\vt', 
\;\; \text{with}\;\; 
\vt'=(\vt\cap\po(n))\oplus\vt_n . 
\end{gather}
\begin{rmk}\label{rem3.1}
 We have 
$
 (\vt+\vt^*)\cap\po(n)=\{Z+Z^*\mid Z\in\vt\}.
$ 
Indeed, if $Z_1,Z_2\in\vt$ and $Z_1+Z_2^*\in\po(n)$, then $Z=
(Z_1+Z_2)/2\in\vt$ and
$Z_1+Z_2^*=Z+Z^*.$
Hence
the maps 
\begin{equation}\label{equa3.2}\begin{cases}
\vt'\ni{Z}\to(Z+Z^*)\in(\vt+\vt^*)\cap\po(n), \\
 \vt'\oplus\mt_0\ni (Z,X)\longleftrightarrow (Z^*+X+Z)\in\pt_0
 \end{cases}
\end{equation}
are $\R$-linear isomorphisms. 
Often we will write $Z\in\vt$ as a sum $Z=Z_0+Z_n,$
where it will be understood that $Z_0\in(\vt\cap\po(n))$ and $Z_n\in\vt_n.$
\end{rmk}
By \eqref{equa3.2}, the Euclidean subspace $\exp(\mt_0)$ is 
a natural candidate for the typical fiber $F_0$
of 
an $\SU(n)$-covariant fibration 
of $\SL_n(\C)/\Vf.$  As we will see, this is in fact
 the case for some important
classes 
of $\Vf$'s.
 \par \smallskip
Being algebraic, $\Vf$ admits the decomposition 
\begin{equation}\label{dcmpv}
 \Vf_0\times\vt'\ni (u,Z_0+Z_n) \longleftrightarrow
  u\cdot\exp(Z_0)\cdot\exp(Z_n)\in\Vf,
\end{equation}
which is a consequence of the Levi-Chevalley decomposition of $\Vf$ and of the polar Cartan decomposition
of $\Vf_{\! r}.$ 
Set 
\begin{equation} \label{defenne}
 N=\{p\in\Mf_n\mid p=v^*v, \;\text{for some}\; v\in\Vf\}.
\end{equation}
\begin{lem}
 The map $v\to v^*v$ defines, by passing to the quotients, an isomorphism 
\begin{equation}\label{tresette}
 \Vf/\Vf_0 \ni [v] \xrightarrow{\;\;\sim\;\;} v^*v\in{N}.
\end{equation}
\end{lem} 
\begin{proof}
 In fact the right action $v\cdot\zetaup=v^*\cdot\zetaup\cdot{v}$ 
 of $\Vf$ on 
 $N$ is transitive and $\Vf_0$ is the stabilizer
 of~\mbox{$e=\Id_n.$}
\end{proof}
\begin{lem}
 The map 
\begin{equation}\label{trecinque}
 \vt'\ni (Z_0+Z_n)\longrightarrow \exp(Z_n^*)\cdot\exp(Z_0)\cdot\exp(Z_n)\in
 N 
\end{equation}
is a diffeomorphism. In particular, $N$ is diffeomorphic to a Euclidean space.
\end{lem} 
\begin{proof}
 In fact, \eqref{trecinque} is smooth and bijective
and its inverse 
can be computed by using the diffeomorphisms $\Vf/\Vf_0\simeq\vt'$ 
of \eqref{dcmpv}, and \eqref{tresette}.
\end{proof}

\begin{lem}\label{tubolar}
We can find a real $r>0$ such that the map 
\begin{equation}\label{treotto} \lambdaup:
  \vt'\times\mt_0\ni(Z_0+Z_n,H)\longrightarrow
  \exp(Z_n^*)\exp(Z_0)\exp(H)\exp(Z_0)\exp(Z_n) \in\Mf_n
\end{equation}
is a  diffeomorphism of $\{\|H\|<r\}$ onto $\{p\in\Mf_n\mid \dist(p,N)<r\}.$
\end{lem} 
\begin{proof} By \eqref{equa3.2}, $\lambdaup$ is a local diffeomorphism at all points where 
it  has an injective differential.
 By using the isometries $p\to{z}^*\! \cdot\!{p}\!\cdot\!{z}$ of $\Ml_n$, 
 we may reduce to points $(0,H)$, 
 where,  
 to compute the differential, we can use the
 Jacobi vector fields $\thetaup_Z$ and $J_X $ 
 on $\gammaup_H$, that where defined in~\S\ref{kiljac}.
Indeed, for $(Z,X)\in\vt'\times\mt_0,$  
 $d\lambdaup(0,H)(Z,0)=\thetaup_Z(1)$  
 and $d\lambdaup(0,H)(0,X)=J_X(1).$
 Moreover, the maps ${\vt'\ni{Z}\to\thetaup_Z(1)\in{T}_{\exp(H)}\Ml_n}$ and 
 ${\mt_0\ni{X}\to{J}_X(1)\in{T}_{\exp(H)}\Ml_n}$ both are injective. 
Thus it suffices to verify that \mbox{$\thetaup_Z(1)\neq{J}_X(1)$} when $Z$ and $X$ are 
not zero. 
 By 
 Lemma~\ref{lemma3.2.13}, 
\begin{equation*}
 \|J_X(1)-\thetaup_Z(1)\|^2\geq \| Z+Z^*\|^2 + 2(H| [Z,Z^*]),
 \quad\forall (Z,X)\in\vt\times\mt_0.
\end{equation*}
For $Z\in\vt',$ we have $\|Z\|=\|Z^*\|\leq\|Z+Z^*\|.$
Thus
\begin{equation*}
 \left|(H|[Z,Z^*])\right| \leq \|H\| \cdot \| Z+Z^*\|^2.
\end{equation*}
This implies that, for some $r>0,$
 \eqref{treotto} defines a local diffeomorphism, and hence a smooth covering,
of $\vt'\times\{\|H\|<r\}$ onto $\{p\in\Mf_n\mid \dist(p,N)<r\}.$ This is in fact a 
global diffeomorphism 
because both spaces are simply connected.
\end{proof} 
Set 
\begin{gather}\label{granvprimo}
 \Vf'=\{\exp(Z_0)\exp(Z_n)\mid Z_0+Z_n\in\vt'\}
 \intertext{and consider the map}\label{treundici}
 \muup: \SU(n)\times\mt_0\times\Vf'\ni (u,X,v)\longrightarrow u\cdot\exp(X)\cdot{v}\in\SL_n(\C).
\end{gather}

\begin{prop}\label{propo3.3.4}
The map \eqref{treundici} is onto. \par 
There is a real $r>0$ for which $\muup$ is a diffeomorphism of $\{\|X\|<r\}$
onto the open manifold $\{\zetaup\in\SL_n(\C)\mid \dist(\zetaup^*\zetaup,N)<2r\}.$
\end{prop} 
\begin{proof} The set
 $N=\{z^*z\mid z\in\Vf\}$ 
 is a properly embedded smooth submanifold of~$\Mf_n$. Hence, for each
$p\in\Mf_n$, there is a $z_p\in\Vf$ with
\begin{equation*}
 \dist(p,z_p^*z_p)=\dist(p,N).
\end{equation*}
The geodesic joining $z_p^*z_p$  to $p$ has the form 
$[0,1]\ni{t}\to\gammaup(t)={z}^*_p\exp(tH)z_p$ for some $H\in\po(n)$, 
and $\dot{\gammaup}(0)$ is orthogonal to $N$ at $z^*_pz_p$. The isometry
$q\to {z_p^*}^{-1}q\,z_p^{-1}$ maps $N$ into itself, $z_p^*z_p$ to $e$ and
$\dot{\gammaup}(0)$ to $H$. Thus $H\in{T}_e\Mf_n=\po(n)$ 
belongs to $\mt_0$.
\par This shows that, 
if $\zetaup\in\SL_n(\C)$ and $z_p^*z_p$ is the nearest point in $N$ 
to $p=\zetaup^*\zetaup,$ 
then
\begin{equation*}
 p=\zetaup^*\zetaup=z_p^*\exp(H)z_p,\;\;\text{for some $z_p\in\Vf'$ and  $H\in\mt_0.$}
\end{equation*}
The matrix
$ u=\zetaup\cdot{z}_p^{-1}\cdot\exp(-H/2)$
belongs to $\SU(n)$. Indeed
\begin{align*}
 u^*u&=\exp(-H/2)\cdot [z_p^{-1}]^*\cdot \zetaup^*\cdot\zetaup\cdot{z}_p^{-1}\cdot\exp(-H/2)\\
 &=\exp(-H/2)\cdot [z_p^{-1}]^*\cdot z_p^*\cdot\exp(H)\cdot{z}_p\cdot{z}_p^{-1}\cdot\exp(-H/2)=\Id_n.
 \end{align*}
 Since $\zetaup=u\cdot\exp(H/2)\cdot{z}_p,$ this proves that \eqref{treundici} is onto.
\par 
The second part of the statement is then
a consequence of Lemma~\ref{tubolar}.
\end{proof}
\begin{cor}
The map \begin{equation}\label{eqq3.38}
\SU(n)\times\mt_0\ni (x,X) \longrightarrow \pi(x\cdot\exp(X))\in\SL_n(\C)
/\Vf,
\end{equation}
where $\pi:\SL_n(\C)\to\SL_n(\C)/\Vf$ is the projection onto the quotient, is onto.
By passing to the quotient, it defines a surjective smooth map 
\begin{equation}
 \SU(n)\times_{\Vf_0}\mt_0 \longrightarrow \SL_n(\C)/\Vf,
\end{equation}
where $\SU(n)\times_{\Vf_0}\mt_0$ is the quotient of $\SU(n)\times\mt_0$ modulo the 
equivalence relation 
\vspace{5pt}
\par\centerline{\qquad\qquad\qquad $(x,X)\sim (x\cdot{u},u^*Xu)$ \; for \;
$x\in\SU(n),$ $X\in\mt_0$ and $u\in\Vf_0.$\qed }
 \end{cor}
 \subsection{Decomposition of \texorpdfstring{$\Kf$}{TEXT}}
Let $\Vf$ be a closed subgoup of the 
complexification $\Kf$ of a compact Lie group $\Kf_0.$
 We can assume that in turn
$\Kf$ is a linear subgroup of $\SL_n(\C)$, with $\Kf_0=\Kf\cap\SU(n),$ and 
$\If_0=\Vf\cap\SU(n)$ a maximal compact subgroup of $\Vf$. 
We obtain:
\begin{prop} With $\ft_0=\mt_0\cap\kt,$ we have the commutative 
diagram with surjective arrows 
\begin{equation} \label{eq:3.48} 
{{\xymatrix{
\Kf_0\times\ft_0 \ar[rr]\ar[rd] &&
\Kf_0\times_{\If_0}\ft_0\ar[ld] \\
& \Kf/\Vf,}}}
\end{equation}
where the horizontal arrow is the projection onto the quotient,
the left one is obtained by restricting \eqref{eqq3.38}, and
the right one by passing to the quotient.
\end{prop} 
We denoted by $\Kf_0\times_{\If_0}\!\ft_0$ the quotient of the product $\Kf_0\times\ft_0$
by the equivalence relation $(x,X)\sim (x\cdot{u},\Ad(u^{-1})(X))$ for $x\in\Kf_0,$
$X\in\ft_0$ and $u\in\Vf_0$. The right arrow maps the equivalence class of $(x,X)$
to $\pi(x\cdot\exp(X))\in\Kf/\Vf\subset\SL_n(\C)/\Vf.$
\begin{proof}
It is sufficient to follow the proof of Proposition~\ref{propo3.3.4}
and check that, for $\zetaup\in\Kf$, we obtain 
$X\in\ft_0$ and $x\in\Kf_0.$ \par 
In fact, 
in this case, $\zetaup^*\zetaup=z^*\exp(2X)z\in\Kf\cap\Pf_0(n),$ with $z\in\Vf,$
implies that $\exp(2X)={z^*}^{-1}\zetaup^*\,\zetaup\,{z}^{-1}\in\exp(\mt_0)\cap\Kf=
\exp(\ft_0)$.
\end{proof}
We have the analogous of Proposition~\ref{propo3.3.4}.
\begin{prop} \label{proptreotto}
The map 
\begin{equation}
 \Kf_0\times\ft_0\times\Vf'\ni (u,X,v)\longrightarrow u\cdot\exp(X)\cdot{v}
 \in\Kf
\end{equation}
is always surjective and there is $r_0>0$ such that, for all $0<r\leq{r}_0,$ 
 it is a diffeomorphism of $\{\|X\|<r\}$
onto a tubular neighborhood of~$M_0=\Kf_0/\Vf_0$ in~$M_-.$ \qed
\end{prop}
It is known that the right arrow in \eqref{eq:3.48} \textit{is}  the Mostow fibration of $\Kf/\Vf$
when $\Vf$ is reductive 
(see e.g.~\cite{Most55,Tum}).  We give here a simple proof relying on the preparation done in
\S\ref{sec3}.
\begin{prop} \label{prop:3.4.6} 
If $\Vf$ is reductive, then the natural surjective map 
\begin{equation}\label{fibrherm}
\Kf_0\times_{\If_0}\ft_0 \to {M}_-=\Kf/\Vf\end{equation}
is a diffeomorphism.  
\end{prop}  
\begin{proof} 
In this case $\Vf,$ being 
algebraic and 
self-adjoint,  
has the Cartan decomposition 
$\Vf=\Vf_0\times\exp(\vt'),$ 
with $\vt'=\vt\cap\pt_0(n)$.
By Lemma~\ref{lem:3.13.3}, the map 
\begin{equation*} \lambdaup_{\kt}:
\vt'
 \times\ft_0\ni (Z,H)\longrightarrow \exp(Z^*)\cdot\exp(H)\cdot\exp(Z)\in 
 \Kf\cap\Pf_0(n) \end{equation*} 
is surjective.
Moreover, it is 
a local diffeomorphism at every point of $\vt'\times\ft_0$. In fact, we can reduce to prove this
fact at points $(0,H)$, where the differential at $(Z,X)$ is $J(1)$ for $J(t)=\thetaup_Z+J_X\in\Jf(H).$
Then $\|J(1)\|\geq\|J(0)\|=2\|Z\|>0$ for $Z\neq{0},$ while $J_X(1)\neq{0}$ if $X\neq{0}.$ 
Since $\kt\cap\po(n)=\vt'\oplus\ft_0$, this proves that $d \lambdaup_{\kt}(0,H)$ is a linear isomorphism. 
Thus, 
being a connected covering of a simply connected space, $\lambdaup_{\kt}$
is  a global diffeomorphism. \par 
Hence, for every $\zetaup\in\Kf$, there is a unique pair
$(Z,H)\in\vt'\times\ft_0$ such that 
\begin{equation*}
 \zetaup^*\cdot\zetaup=\exp(Z^*)\cdot\exp(H)\cdot\exp(Z);
\end{equation*}
then
$u=\zetaup\cdot\exp(-Z)\cdot\exp(-\tfrac{1}{2}H)\in\Kf_0$ and we obtain the 
direct product decomposition 
\begin{equation}\label{dec3.50}
 \Kf=\Kf_0\cdot \exp(\ft_0)\cdot\exp(\vt'),
\end{equation}
from which the statement follows.
\end{proof} 
The complex $\Kf$-homogeneous $M_-$ of
Proposition~\ref{prop:3.4.6} corresponds to 
an $M_-$ which
is  the Stein complexification of a totally real $\Kf_0$-homogeneous compact  $M_0$. 
An $M_0$ having a positive $CR$ dimension corresponds to 
a $\Vf$ having a nontrivial unipotent radical. \par 
Before investigating cases where, even though $\vt_n\neq{0},$ 
 \eqref{fibrherm} is nevertheless a diffeomorphism,
 we observe that, 
when we know that 
decomposition \eqref{treundici} is unique,
we can extract some extra information from the minimal distance
 characterization of $z^*_pz_p$ in the proof of Proposition~\ref{propo3.3.4}.
 For instance,
as a corollary of  Proposition~\ref{propo3.3.4}, we obtain the following
\begin{prop}
For $h\in\Pfo(n)$, 
 denote by $D_\ell(h)$ 
 the minor determinant of the first $\ell$ rows and columns of
 $h$. Set $D_0(h)=1$ 
 and let $0<\lambdaup_1(h)\leq\cdots\leq\lambdaup_n(h)$ be the eigenvalues of
 $h$. Then 
\begin{equation}
 \dist(h,e)={\sum}_{\ell=1}^n|\log(\lambdaup_\ell(h))|^2 
 \geq {\sum}_{\ell=1}^n |\log(D_\ell(h)/D_{\ell-1}(h))|^2.
\end{equation}
If $h$ is not diagonal, we have strict inequality.
\end{prop} 
\begin{proof}
 We take $\Vf$ equal to the group of unipotent upper triangular matrices in $\GL_n(\C)$.
 The element $\deltaup=e^{\Delta}\in{N}_h=\{z^*{h}z\mid z\in\Vf\}$, with $\Delta\in\po$, 
 at minimal distance from
 $e$ satisfies $\trac([Z+Z^*]\Delta)=0$ for all nilpotent upper triangular $Z$ and hence is diagonal.
 The unique diagonal $
 \deltaup= z^*hz$ in $N_h$ 
 is the one obtained by the Gram-Schmidt 
 orthogonalization procedure. The proof is complete.
\end{proof}
The orbit of a point $p\in\Ml_n$ by the group of unipotent upper triangular
matrices of $\SL_n(\C)$ is an example of a \emph{horocycle}
of maximal dimension 
in a symmetric space of noncompact type. We will generalize this
situation while 
outlining a class of subroups $\Vf$ for which $F_0=\exp(\ft_0)$
can be taken as the fiber of the $\Kf_0$-covariant fibration. 
\par
Following  \cite[p.17]{Ga72}, 
we call \textit{horocyclic} in $\kt$ the nilpotent subalgebras 
which are nilradicals of  parabolic subalgebras of $\kt.$
\begin{lem}\label{lemmatreundici}
Let $\qt$ be a parabolic subalgebra of $\slt_n(\C),$ with nilradical $\qt_n.$ 
Assume that $\qt\cap\qt^*$ is a reductive Levi factor
of $\qt$. Let $H\in\qt\cap\po(n).$ Then, for $Z_0\in\qt\cap\qt^*,$ $T\in\Ct_0(H)\cap\qt$ and
$Z_n\in\qt_n$ the Jacobi vector fields $J_1=\thetaup_{Z_0}+t\thetaup_T$ and $J_2=\thetaup_{Z_n}$
are orthogonal at all points of $\gammaup_H.$  
\end{lem} 
\begin{proof}
We show, separately, that $\thetaup_{Z_0}$ and $\thetaup_T$ are both orthogonal to $\thetaup_{Z_n}$ 
at all points of $\gammaup_H.$ We have 
\begin{align*}
 (\thetaup_T(t)|\thetaup_{Z_n}(t)) &= \trac\big(2T{e}^{-tH}(e^{tH}Z_n^*+Z_n e^{tH})\big)=
 2\trac(TZ_n+TZ_n^*)=0,\\
 (\thetaup_{Z_0}(t)|\thetaup_{Z_n}(t)) &=\trac\big((e^{-tH}Z^*_0+Z_0{e}^{-tH})(e^{tH} Z_n+Z_n^*e^{tH})\big)\\
 &=\trac\big( Z_0^* (e^{tH}Z_n e^{-tH}) + Z_0^*Z_n^*+Z_0Z_n+(e^{tH}Z_0e^{-tH})Z_n^*\big)=0
\end{align*}
because  $\qt\cap\qt^*$ and $\qt_n$ are orthogonal for the trace form of 
the canonical representation of
$\slt_n(\C).$ Indeed, 
The expression in the last line is twice the sum of the real parts of the  product of $Z_0$ and
$Z_n$ and of $e^{-tH}Z_0^*e^{tH}\in\qt\cap\qt^*$ and $Z_n.$
\end{proof}
\begin{prop}\label{proptredodici}
 If $\vt_n$ is horocyclic in $\kt$, then 
 \begin{equation}\label{trediciotto}
 \Vf'\times\ft_0\ni (v,H)\longrightarrow v^*\exp(H)\, v\in
 \Mf(\Kf)
 =\Pf_0(n)\cap\Kf
\end{equation}
is a diffeomorphism.
\end{prop} 
\begin{proof}
In fact, we can find a parabolic $\qt$ in $\slt_n(\C)$ such that $\qt\cap\qt^*$ is its reductive Levi factor and
$\vt_n=\qt_n\cap\kt.$ Then we can reduce to proving the proposition in the
case where $\Kf=\SL_n(\C)$ and $\ft_0=\mt_0.$ We want to show 
that \eqref{treotto} is a local diffeomorphism. To this aim, with the notation
of \S\ref{kiljac}, it suffices 
 to prove that, for $Z\in\vt$ and $H,X\in\mt_0,$
we have $\thetaup_Z(1)\neq{J}_X(1)$ when $Z+X\neq{0}.$ We split 
$Z$ into the sum $Z=Z_0+Z_n,$ with $Z_0\in\vt\cap\po(n)$ and
$Z_n\in\vt_n.$ Then the fact that 
$\thetaup_{Z_0}(1)+J_X(1)\neq{0}$ if $Z_0+X\neq{0}$ follows
from Lemma~\ref{lem:3.13.3} because of
Lemma~\ref{lemmatreundici}. Hence \eqref{treotto} is a connected covering
of a simply connected manifold and thus a global diffeomorphism.
\end{proof}
Proposition~\ref{proptredodici} can be slightly generalized. It was
shown in 
\cite[p.251]{MN05} that there is a unique maximal complex Lie subalgebra
$\wt$ of $\kt$ with $\vt\subseteq\wt\subseteq\vt+\overline{\vt}.$
The $CR$-algebra 
$(\kt_0,\vt)$ and the corresponding $\Kf_0$-homogeneous
$CR$ manifold $M_0$ are called \emph{weakly nondegenerate} when
$\wt=\vt.$ 
If this is not the case, 
$M_0$ turns out to be 
the total space
of a complex
$CR$-bundle  with no%
ntrivial fibers
over a weakly nondegenerate $\Kf_0$-homogeneous $CR$ 
manifold~$M'_0,$ having $CR$ algebra $(\kt_0,\wt).$
\begin{prop}\label{proptretredici}
 Let $\wt$ be the largest complex Lie algebra with $\vt\subseteq\wt\subseteq
 \vt+\overline{\vt}.$ If $\wt_n=\nt(\wt)$ is horocyclic in $\kt,$ then
 \eqref{trediciotto} is a diffeomorphism.
\end{prop} \begin{proof}
As above, we reduce the proof to the case where $\Kf=\SL_n(\C).$ 
The proof follows the same pattern of the proof 
of Proposition~\ref{proptredodici}. 
We denote by $\qt$ a parabolic Lie subalgebra of
$\slt_n(\C)$ with $\qt_n=\wt_n$ 
and use the notation of \S\ref{kiljac}.
We need to prove that, for
$Z\in\vt'=(\vt\cap\po(n))\oplus\vt_n$ and $X,H\in\mt_0,$ we have
$\thetaup_Z(1)+{J}_X(1)\neq{0}$ if $Z+X\neq{0}$. 
To this aim it is convenient to split $Z$ into a sum $Z=U+W,$ 
with $U\in\vt'\cap\wt\cap\overline{\wt}$ and $W\in\qt_n.$ 
Let us consider first
$J=\thetaup_{U}+J_X$. We note that $\dot{J}(0)=X+\tfrac{1}{2}[X,U-U^*]$
is orthogonal to $J(0)=U+U^*.$ Indeed $(X|U+U^*)=0$ because
$\wt+\overline{\wt}=\vt+\overline{\vt}$ and,
since $[U,U^*]\in\wt\cap\po(n),$
\begin{equation*}
(U+U^*|[H,U-U^*])=\trac([H,U-U^*](U+U^*))=2\trac(H\cdot [U,U^*])=0.
\end{equation*}
By Lemma~\ref{lem:3.13.3}, this implies that $J(1)\neq{0}$ if 
$Z+X\neq{0}.$ Finally, we note that $\thetaup_W(0)$ and
$\dot{\thetaup}_W(0)$ are orthogonal to both $J(0)$ and $\dot{J}(0)$
to conclude, using again Lemma~\ref{lem:3.13.3}, that 
$J_Z(1)+J_X(1)=J(1)+J_W(1)\neq{0}$
when $X+Z=(X+U)+W\neq{0}.$\par 
This shows that \eqref{trediciotto}, being a connected smooth covering
of a simply connected manifold, is a global diffeomorphism.
\end{proof}
By using the argument in the proof of Proposition~\ref{prop:3.4.6},
we conclude:
 \par 
 \begin{thm}\label{thm3.14}
 Let $\wt$ be the largest complex Lie algebra with $\vt\subseteq\wt\subseteq
 \vt+\overline{\vt}.$ If $\wt_n=\nt(\wt)$ is horocyclic in $\kt,$ then
 \eqref{fibrherm} is a global diffeomorphism and therefore we obtain the
 $\Kf_0$-equivariant
  Mostow fibration of $M_-$ over $M_0$
\begin{equation}
\xymatrix{\Kf_0\times_{\Vf_0}\ft_0 \ar[rr] 
\ar[dr] && M_-\ar[dl]\\
& M_0}
\end{equation}
with Hermitian fiber.\qed
 \end{thm}

We keep the notation of \S\ref{subsreductive} and denote by $\wt$
the largest Lie subalgebra of~$\kt$~with 
\begin{equation}
\vt\subseteq\wt\subseteq\vt+\overline{\vt}.
\end{equation}
\begin{defn}\label{defhnr}
 We say that 
 $(\kt_0,\vt)$ 
 is {\HNR} 
 if $\wt_n=\nt(\wt)$
is horocyclic.
\end{defn}
For further reference, we reformulate the result obtained so far in the
following form. 
\begin{thm}\label{teotrequindici} 
If 
$(\kt_0,\vt)$ 
is {\HNR}, then we have the direct product decomposition
\begin{equation}\label{eqtreventuno} \vspace{-19pt}
\Kf=\Kf_0\cdot \exp(\ft_0)\cdot \Vf'.
\end{equation}
\qed
\end{thm} 
\begin{exam}\textsl{Minimal orbit of \texorpdfstring{$\SU(2,2)$ in $\Fi_{1,2}(\C^4)$}{TEXT}.} \par
We fix in $\C^4$ the Hermitian form associated to the matrix 
\begin{equation*} 
\begin{pmatrix}
 \Id_2\\
 & -\Id_2
\end{pmatrix}.
\end{equation*}
We let the corresponding group $\SU(2,2)$ operate on the flag manifold $\Fi_{1,2}(\C^4)$,
consisting of the pairs $(\ell_1,\ell_2)$ of a line $\ell_1$ and a $2$-plane $\ell_2$ with
$0\in\ell_1\subset\ell_2\subset\C^4.$ The minimal orbit is 
\begin{equation*}
 M_0=\{(\ell_1,\ell_2)\mid \ell_1\subset\ell_2=\ell_2^\perp\},
\end{equation*}
where the orthogonal is taken with respect to the fixed Hermitian form. 
It is the total space of a $\CP^1$-bundle over a smooth real manifold 
and in particular is  Levi-flat of $CR$ dimension $1.$
With 
 $\Kf_0=\Sb(\Ub(2)\times\Ub(2)),$  $\Kf=\Sb(\GL_2(\C)\times\GL_2(\C))$,
 the stabilizer 
\begin{equation*}
 \Vf=\left.\left\{ 
\begin{pmatrix}
 a \\
 & a
\end{pmatrix}\right| a\in
\Sb\Tb^+_2(\C)\right\}
\end{equation*}
of the base point  
$\op=(\langle e_1+e_3\rangle,\langle e_1+e_3,e_2+e_4\rangle)$
(here $\Tb^+_2(\C)$ is the group of upper triangular $2\times{2}$
complex matrices with non vanishing determinant and $\Sb\Tb^+_2(\C)$ its normal subgroup
consisting of those having determinant $1$)
has Lie algebra 
\begin{equation*}
 \vt=\left.\left\{ 
\begin{pmatrix}
 \lambdaup& \alphaup\\
 0 & -\lambdaup\\
 && \lambdaup & \alphaup \\
 && 0 & -\lambdaup
\end{pmatrix}\right| \lambdaup,\alphaup\in\C\right\}.
\end{equation*}
Clearly $\vt_n$ is not horocyclic. 
We note that \begin{equation*} \wt=
\vt+\overline{\vt}=\left.\left\{ \begin{pmatrix}
X & 0\\
0 & X
\end{pmatrix}\right| X\in\slt_2(\C)\right\}=\vt'
\end{equation*}
is a complex Lie algebra. Thus, although $\Vf$ is not {\HNR},
nevertheless we have a Mostow fibration with Hermitian fibers by
Theorem~\ref{thm3.14}.
\end{exam}
\begin{rmk} Example~\ref{esampduesette} shows that \eqref{trediciotto}
is not, in general, a diffeomorphism when $(\kt_0,\vt)$ is not
{\HNR}.
\end{rmk}
\section{Mostow fibration in general and the \texorpdfstring{\HNR}{TEXT} condition}
\label{sec5}
\subsection{The set \texorpdfstring{$\Pp_0(\vt)$}{TEXT}}
To better understand the notion introduced in Definition~\ref{defhnr}
and to characterize the fiber of the Mostow fibration of $M_-$ on $M_0$
in general, 
 it is convenient to rehearse some notions that were introduced 
in \cite[\S{3}]{AMN2013}.
We simply assume, at the beginning, that $\kt$ is any reductive Lie algebra over~$\C$.
\par
For a Lie subalgebra $\at$ of $\kt$,  let us denote 
by 
$\nt(\at)$ the ideal consisting of the $\ad_{\kt}$-nilpotent
elements of its radical. 
Starting from any splittable
Lie subalgebra $\vt$ of $\kt$ we construct
a sequence $\{\vt_{(h)}\}$ of Lie subalgebras 
by setting recursively 
\begin{equation} 
\begin{cases}
 \vt_{(0)}=\vt,\\
 \vt_{(h+1)}=\mathbf{N}_{\kt}(\nt(\vt_{(h)}))=\{Z\in\kt\mid [Z,\nt(\vt_{(h)})]\subset\nt(\vt_{(h)})\}, & \forall h\geq{0}.
\end{cases}
\end{equation}
Each 
$\vt_{(h)}$, with $h\geq{1}$, 
is the normalizer in $\kt$ of the ideal of $\ad_{\kt}$-nilpotent elements
of the radical of $\vt_{(h-1)}.$
It was shown in \cite{AMN2013} that 
$\vt_{(h)}\subseteqq\vt_{(h+1)}$ and 
$\nt(\vt_{(h)})\subseteqq\nt(\vt_{(h+1)})$ for all $h\geq{0},$ 
and that the union $\et={\bigcup}_{h\geq{0}}\vt_{(h)}$ 
is a parabolic subalgebra   of $\kt,$ with $\vt\subset\et$ and $\nt(\vt)=\vt_n\subset\nt(\et).$
We call $\et$  
the \emph{parabolic regularization of $\vt.$}
Hence 
\begin{equation}
\Pp(\vt)=\{\qt\mid \qt\;\text{is parabolic in $\kt$ and
$\vt\subset\qt$, $\nt(\vt)\subset\nt(\qt)$}\}
\end{equation}
is nonempty. 
Let us prove a general simple lemma on parabolic Lie subalgebras.
\begin{lem}\label{lemma4.1}
If $\qt_{{1}},\qt_2$ are parabolic Lie subalgebras
of $\kt$, then the Lie subalgebra $\qt=\qt_{{1}}\cap\qt_2+\nt(\qt_{{1}})$
is parabolic in $\kt$.
\end{lem} \begin{proof} We know (see e.g. 
\cite[Ch.VIII,Prop.10]{Bou75}) that $\qt_{{1}}\cap\qt_2$ contains
a Cartan subalgebra $\hg$ of $\kt$. If  $\Rad$ is the  corresponding set of roots, 
then each 
$\qt_i$ ($i=1,2$)
decomposes into a  
direct sum 
\begin{equation*}
\qt_i=\hg\oplus{\sum}_{\begin{smallmatrix}\alphaup\in\Rad,\\
\alphaup(A_i)\geq{0}\end{smallmatrix}}\kt_{\alphaup},
\end{equation*} 
where $A_1,A_2\in\hg_{\R}$ and, for each $\alphaup\in\Rad,$
 $\kt_{\alphaup}
=\{Z\in\kt\mid [A,Z]=\alphaup(A)Z,\;\forall A\in\hg_{\R}\}$ is the root space of $\alphaup.$ 
\par 
Take $\epsilon>0$ so small that 
\, $\epsilon\cdot|\alphaup(A_2)|<\alphaup(A_1)$ if $\alphaup(A_1)>0$. Then
\begin{equation*}
\qt=\hg\oplus{\sum}_{\begin{smallmatrix}
\alphaup\in\Rad,\\
\alphaup(A_1+\epsilon{A}_2)>0
\end{smallmatrix}}\kt_{\alphaup},
\end{equation*} 
is parabolic. In fact, if
$\Li(\qt_i)$ are the $\hg$-invariant reductive summands 
of $\qt_i$ and
$\nt(\qt_i)$ the ideals of nilpotent elements of their radicals, we have
\begin{equation*} \vspace{-24pt}
\qt=(\Li(\qt_{{1}})\cap\Li(\qt_2))\oplus (\Li(\qt_{{1}})\cap\nt(\qt_2))\oplus
\nt(\qt_{{1}}).
\end{equation*}
\end{proof}
From now on we assume that $\kt$ is the complexification of its compact 
real form $\kt_0$. 
Conjugation in $\kt$ will be understood with respect to  $\kt_0.$
Using  parabolic regularization 
and Lemma~\ref{lemma4.1} 
we obtain 
\begin{prop} If $(\kt_0,\vt)$ is $\nt$-reductive, then 
$\Pp(\vt)$ contains a $\qt$ having a conjuga\-tion-invariant 
reductive Levi subalgebra. 
\end{prop} \begin{proof}
We can take $\qt=(\et\cap\overline{\et})+\nt(\et)$, for the parabolic
regularization $\et$ of~$\vt$.
\end{proof}
This shows that, for an $\nt$-reductive $(\kt_0,\vt),$ the set 
\begin{equation}
\Pp_0(\vt)=\{\qt\in\Pp(\vt)\mid \qt=(\qt\cap\overline{\qt})\oplus\nt(\qt)\}
\end{equation}
is nonempty. 
For $\qt\in\Pp_0(\vt)$ we will use $\Li(\qt)=\qt\cap\overline{\qt}$. 
The parabolic regularizazion produces a \textit{small} $\et$
and a corresponding smaller $(\et\cap\overline{\et})\oplus\nt(\et)$
in $\Pp_0(\vt)$. We are however  more interested in the \textit{maximal} elements of $\Pp(\vt)$.
To explain the meaning of maximality, 
 we prove
(cf.~\cite[Proposition 20]{AMN2013}) 
\begin{prop}
 If $(\kt_0,\vt)$ is $\nt$-reductive and 
 $\qt$ any maximal element of $\Pp_0(\vt)$, then 
\begin{equation}
 \qt=\Lie\big(\nt(\vt)+\Li(\qt)\big) \quad\text{and}\quad 
 \nt(\qt)={\sum}_h\ad^h(\Li(\qt))(\nt(\vt)).
\end{equation}
\end{prop} \begin{proof} Let $\qt\in\Pp_0(\vt)$ and denote by $\zt$ the center of $\Li(\qt)$.
Being invariant under conjugation, it is the complexification
of the Lie subalgebra $\zt_0$ of a maximal torus $\td_0$ of $\kt_0$. 
Set $\zt_{\R}=\mathpzc{i}\!\zt_0$. Following the construction of Konstant in
\cite{Kos10}, we  consider the set $\Zf$ 
consisting 
of the nonzero elements
$\nuup$ of the dual $\zt_{\R}^*$ for which 
\begin{equation*}
\kt_{\nuup}=\{X\in\kt\mid [Z,X]=\nuup(Z)X,\;\forall Z\in\zt_{\R}\}\neq
\{0\}.
\end{equation*}
This set $\Zf$ shares many properties of the root system of a semisimple
Lie algebra.
With the scalar product defined on $\zt_{\R}$ by the restriction
of the trace form of a faithful linear representation of $\kt$
and the corrisponding
dual scalar product on $\zt_{\R}^*$,
we have \begin{align}
\tag{$i$} & \nuup\in\Zf\Longrightarrow -\nuup\in\Zf,
\;\;\text{and}\;\; \overline{\kt}_{\nuup}=\kt_{-\nuup},\\
\tag{$ii$} & \nuup_1,\nuup_2,\nuup_1+\nuup_2
\in\Zf\Longrightarrow [\kt_{\nuup_1},
\kt_{\nuup_2}]=\kt_{\nuup_1+\nuup_2},\\
\tag{$iii$} &\nuup_1,\nuup_2\in\Zf\;\text{and}\; (\nuup_1|\nuup_2)>0
\Longrightarrow \nuup_1-\nuup_2\in\Zf,\\
\tag{$iv$} & \forall \nuup\in\Zf,\;\kt_{\nuup}\;\text{is
an irreducible $\Li(\qt)$-module,}\\
\tag{$v$} & \nt(\qt)={\sum}_{\nuup>0}\kt_{\nuup},\;\text{for
some lexicographic order in $\Zf,$}\\
\tag{$vi$} & \exists \text{\; a basis $\{\muup_1,\hdots,
\muup_{\ell}\}\subset\Zf$ of positive simple roots of $\zt_{\R}^*$.}
\end{align} 
\par
The Lie subalgebra $\Lie(\nt(\vt)+\Li(\qt))$ is contained in $\qt$
and is a direct sum 
\begin{equation*}
 \Lie(\nt(\vt)+\Li(\qt))=\Li(\qt)\oplus{\sum}_{\nuup\in\mathpzc{E}}\kt_{\nuup},
\end{equation*}
for a subset $\mathpzc{E}$ of $\Zf^+=\{\nuup>0\}$.
Assume 
that there is a positive simple root $\muup_i$ 
which 
does not belong to $\mathpzc{E}$.
Since $\muup_i$ is simple,  
$\qt'=\qt\oplus\kt_{-\muup_i}$ is still a parabolic Lie subalgebra. 
Let us show that it is 
an element of $\Pp_0(\vt)$. 
We have \begin{equation*}
\qt'=\Li(\qt')\oplus\nt(\qt'),\;\;\text{with}\;\;
\Li(\qt')=\Li(\qt)\oplus\kt_{\muup_i}\oplus\kt_{-\muup_i}\;\;
\text{and}\;\; \nt(\qt')={\sum}_{\nuup\in(\Zf^+\setminus\{\muup_i\})}
\kt_{\nuup}.
\end{equation*}
Note that 
$\Li(\qt')=\qt'\cap\bar{\qt}'$. An element $X\in\nt(\vt)$
can be written in a unique way as a sum $X={\sum}_{\nuup\in\mathpzc{E}}X_{\nuup}$
with $X_{\nuup}\in\kt_{\nuup}$. Then
$X\in\nt(\qt')$, because $\mathpzc{E}\subset\Zf^+\setminus\{\muup_i\}.$
This shows that $\nt(\vt)\subset\nt(\qt')$, i.e that $\qt'\in\Pp_0(\vt)$.
Thus, if $\qt$ is maximal in $\Pp_0(\vt)$, then $\Lie(\nt(\vt)+\Li(\qt))$ contains all
$\kt_{\muup_i}$ for $i=1,\hdots,\ell$ and thus is equal to $\qt$, because
$(ii)$ and the fact that every positive root is a sum o simple positive roots yield that
$\Lie({\sum}_{i=1}^\ell\kt_{\muup_i})=\nt(\qt).$ Finally, it follows from the discussion above
that $\nt(\qt)$ is the $\ad(\Li(\qt))$-module generated by $\nt(\vt)$.
\end{proof}
Analogously, we obtain 
\begin{prop}
 If $\qt$ i
 s any maximal element of $\Pp(\vt)$, then 
\begin{equation}
 \qt=\Lie(\nt(\vt)+\Li(\qt)),
\end{equation}
for any reductive Levi factor $\Li(\qt)$ of $\qt$, and $\nt(\qt)$ is the
$\ad(\Li(\qt))$-module generated by $\nt(\vt)$.
\qed
\end{prop}
\subsection{A remark on the \texorpdfstring{\HNR}{TEXT} condition} 

Assume that $(\kt_0,\vt)$ is $\nt$-reductive and let
$\Qf$ be the parabolic subgroup of $\Kf$ corresponding to a $\qt$ in $\Pp_0(\vt)$.
Let $\Qf_n$ be the unipotent radical of $\Qf$ and set
$\Vf'=\Vf\cdot\Qf_n.$ Then $\Vf'\cap\overline{\Vf}'=\Vf\cap\overline{\Vf}$ and therefore
the minimal $\Kf_0$ orbits in $M_-=\Kf/\Vf$ and $M'_-=\Kf/\Vf'$ are diffeomorphic as
$\Kf_0$-homogeneous manifolds: 
the $CR$ algebras $(\kt_0,\vt)$ and $(\kt_0,\vt+\qt_n)$
 define two $CR$ structures on the same $M_0=\Kf_0/\Vf_0,$
the latter being \textit{stronger} than the first. 
These are the $CR$ structures inherited
from the embeddings $M_0\hookrightarrow{M}_-$ and $M_0\hookrightarrow{M}'_-.$ 
Note that $M'_{-}$ is the basis of a complex fiber bundle
 $M_-\to{M}'_-,$ with Stein fibers
bi-holomorphic to $\C^k$ for some nonnegative integer
 $k$ (cf. \cite[Thm.30]{AMN2013}). 
 The choice of a maximal $\qt$ in $\Pp_0(\vt)$ leads to
a \textit{minimal} $\vt+\qt_n$, while a minimal $\qt\in\Pp_0(\vt)$ to a \textit{maximal}
$\vt+\qt_n$, defining, when $(\kt_0,\vt)$ is not {\HNR}, 
a maximal $\Kf_0$-homogeneous $CR$ structure on $M_0$ which is {\HNR} and 
stronger than the original one.\par 
\begin{exam}\label{esuquattrocinque}
\textsl{Minimal orbit of \texorpdfstring{$\SU(2,3)$ in $\Fi_{1,3}(\C^5)$}{TEXT}.} \par
We denote by $\Fi_{1,3}(\C^5)$ the flag manifold consisting of the pairs $(\ell_1,\ell_3)$ of a line
$\ell_1$ and a $3$-plane $\ell_3$ of $\C^5$ with $0\in\ell_1\subset\ell_3.$
We fix the Hermitian symmetric form of signature $(2,3)$ in $\C^n,$ 
corresponding to the matrix 
\begin{equation*}
 \begin{pmatrix}
 \Id_2\\
 & -\Id_3
\end{pmatrix},
\end{equation*}
and consider the minimal orbit
for the action of the real Lie group $\SU(2,3)$ in $\Fi_{1,3}(\C^5):$
 \begin{equation*}
M_0=\{(\ell_1,\ell_3)\in\Fi_{1,3}(\C^5)\mid \ell_1\subset\ell_3^\perp\subset
\ell_3\}. 
\end{equation*}
Fix on $M_0$ the base point $\op=(\langle e_1+e_3\rangle, \langle
e_1+e_3, \, e_2+e_5,\, e_5\rangle).$ 
Its stabilizer in~$\Kf$~is 
\begin{equation*}
\Vf=\left.\left\{ \begin{pmatrix}
\lambdaup_1&z_1\\
0 & \lambdaup_2\\
& & \lambdaup_1 & 0 & z_1\\
& & 0 &\lambdaup_3& z_2\\
& & 0 & 0 & \lambdaup_2
\end{pmatrix}\right| \lambdaup_i,z_i\in\C,\;\; \lambdaup_1^2
\cdot\lambdaup_2^2\cdot\lambdaup_3=1\right\},
\end{equation*}
with Lie algebra
\begin{equation*}
\vt=\left.\left\{ \begin{pmatrix}
\lambdaup_1&z_1\\
0 & \lambdaup_2\\
& & \lambdaup_1 & 0 & z_1\\
& & 0 &\lambdaup_3& z_2\\
& & 0 & 0 & \lambdaup_2
\end{pmatrix}\right| \lambdaup_i,z_i\in\C,\;\; 2\lambdaup_1
+2\lambdaup_2+\lambdaup_3=0\right\}.
\end{equation*}
The normalizer of $\vt_n$ in $\kt$ is the parabolic 
 \begin{equation*}
\qt=\left.\left\{ \begin{pmatrix}
\lambdaup_1&z_1\\
0 & \lambdaup_2\\
& & \lambdaup_3 & \alphaup_1 & z_2\\
& & \alphaup_2 &\lambdaup_4& z_3\\
& & 0 & 0 & \lambdaup_5
\end{pmatrix}\right| \lambdaup_i,z_i,\alphaup_i\in\C,\;\; {\sum}_{i=1}^5
\lambdaup_i=0\right\},
\end{equation*}
which is also a maximal element in $\Pp_0(\vt)$ 
and hence $(\st(\ut(2)\times\ut(3)),\vt)$  is not \HNR. \par 
The Lie algebra 
\begin{equation*} \tilde{\vt}=
 \vt+\qt_n=\left.\left\{ 
\begin{pmatrix}
 \lambdaup_1 & z_1 \\
 0 & \lambdaup_2 \\
 && \lambdaup_1 & 0 & z_2\\
 && 0 & \lambdaup_3 & z_3\\
 && 0 & 0 & \lambdaup_2
\end{pmatrix}\right| \lambdaup_i,z_i\in\C,\;\; 2(\lambdaup_1+\lambdaup_2)+\lambdaup_3=0\right\}
\end{equation*}
is the Lie algebra of the stabilizer $\tilde{\Vf}$ in $\Kf=\Sb(\GL_2(\C)\times\GL_3(\C))$ of
$p_0'\in\Fi_{1,2,4}(\C^5)$ for $p_0'=(\langle e_1+e_3\rangle, \langle e_1+e_3,e_4 \rangle,
\langle e_1,e_3,e_4,e_2+e_5\rangle)$. This corresponds to the intersection of
the $\SU(2,3)$-orbit 
\begin{equation*}
M'_+=\{(\ell_1,\ell_2,\ell_4)\in\Fi_{1,2,4}(\C^5)\mid \ell_1=
\ell_2\cap\ell_2^\perp,\; \dim(\ell_4\cap\ell_4^\perp)=1\}
\end{equation*}
with its Matsuki dual $\Kf$-orbit
$M_-'$. With $L_2=\langle e_1,e_2\rangle$ and $L_3=\langle e_3,e_4,e_5\rangle,$
 we have
\begin{equation*}
 M'_-=\left\{\begin{aligned} (\ell_1,\ell_2,\ell_4)\quad\;\;
 \\
 \in\Fi_{1,2,4}(\C^5) \end{aligned}
 \left|\begin{gathered}
  \dim(\ell_1\cap{L}_2)=0,\; \dim(\ell_1\cap{L}_3)=0,\;
 \dim(\ell_2\cap{L}_2)=0,\\ \dim(\ell_2\cap{L}_3)=1,\; \dim(\ell_4\cap{L}_2)=1,\;
 \dim\ell_4\cap{L}_3=2\end{gathered}\right\}\right..
\end{equation*}
This shows that, in this case, the strengthening of the $CR$ structure on $M_0$ corresponds
to considering the compact intersection with its Matsuki dual of an intermediate orbit
in some complex flag manifold of the same complex semisimple Lie group (in this case
of $\SL_5(\C)$).
\end{exam}

\begin{prop}
 Assume that $(\kt_0,\vt)$ is $\nt$-reductive. Then, if 
 $\wt$ is a 
 complex Lie subalgebra of $\kt$ with $\vt\subseteq\wt\subseteq\vt\oplus\overline{\vt},$
 then also $(\kt_0,\wt)$ is $\nt$-reductive.
\end{prop} 
\begin{proof} The reductive Lie group $\kt$ has an invariant 
nondegenerate bilinear form $\betaup,$ which is real and negative definite
on $\kt_0.$ We observe that, if the pair $(\kt_0,\vt)$ is $\nt$-reductive,
then $\vt_n=\vt\cap\vt^\perp,$ where $\vt^\perp=\{Z\in\kt\mid\betaup(Z,Z')=0,\;\forall Z'\in\vt\},$
and that $\vt+\overline{\vt}$ has the direct sum decomposition
\begin{equation*}
\vt+\overline{\vt}=\vt\oplus\overline{\vt}_n.
\end{equation*}   
If $\wt$ is a complex Lie subalgebra with
 $\vt\subseteq\wt\subseteq\vt+\overline{\vt}$,
then $\wt=\vt\oplus (\wt\cap\overline{\vt}_n).$ Since $\betaup$ defines
a duality pairing between
$\vt_n$ and
$\overline{\vt}_n,$ we obtain the decomposition \begin{equation*}
\wt=(\wt\cap\overline\wt)\oplus\wt_n,\;\;\text{with}\;\;
 \wt_n=\vt_n\cap(\wt\cap\overline{\vt}_n)^\perp,\;\;
 \wt\cap\overline{\wt}=(\vt\cap\overline{\vt})
 \oplus(\vt_n\cap\overline{\wt})
 \oplus(\overline{\vt}_n\cap\wt),
\end{equation*}
showing that also $(\kt_0,\wt)$ is $\nt$-reductive.
\end{proof}
\begin{rmk}
 If $(\kt_0,\vt)$ is $\nt$-reductive, then $\vt$ is the Lie algebra of an algebraic Lie subgroup $\Vf$
 of $\Kf.$ This is the content of \cite[Thm.26]{AMN2013}. In particular, all Lie subalgebras
 $\wt$ with $\vt\subseteq\wt\subseteq\vt+\overline{\vt}$ are $\Lie(\Wf)$ for an algebraic
 Lie subgroup $\Wf$ of~$\Kf.$
\end{rmk}
\begin{exam}\textsl{Minimal orbit of \texorpdfstring{$\SU(2,3)$ in $\Fi_{1 ,2}(\C^5).$}{TEXT}}\par
We partly use the notation of Example~\ref{esuquattrocinque}.
Denote by ${M}_0$ the minimal orbit of $\SU(2,3)$ in the flag $\Fi_{1,2}(\C^5)$ of nested
lines and $2$-planes. 
\begin{equation*}
 M_0=\{ (\ell_1,\ell_2\in\Fi_{1,2}(\C^5)\mid \ell_2\subset\ell_2^\perp\}
\end{equation*}
is a $CR$ manifold of type $(3,4).$
It is the total space of a $\CP^1$-bundle on the $CR$ manifold 
$M'_0$ of isotropic $2$-planes in 
the Grassmannian $\Gr_2(\C^4),$  which has type $(2,4).$ 
The stabilizer ${\Vf}$ of the base point
$p_0=(\langle e_1+e_3\rangle, \langle e_1+e_3, e_2+e_4\rangle),$
has Lie algebra
\begin{equation*}
{\vt}=\left.\left\{ 
\begin{pmatrix}
 \lambdaup_1 & z_1\\
 0 & \lambdaup_2 \\
 && \lambdaup_1 & z_1& z_2\\
 && 0 & \lambdaup_2 & z_3 \\
 && 0 & 0 & \lambdaup_3
\end{pmatrix}\right| \begin{aligned}&\lambdaup_i,z_i\in\C\\
& 2\lambdaup_1+2\lambdaup_2+\lambdaup_3=0\end{aligned}\right\}.
\end{equation*}
The largest $\qt\in\Pp_0(\vt)$ has 
\begin{equation*}
 \qt_n=\left.\left\{ 
\begin{pmatrix}
0 & z_1\\
 0 & 0 \\
 && 0 & z_2& z_3\\
 && 0 & 0 & z_4 \\
 && 0 & 0 & 0
\end{pmatrix}\right| \begin{aligned}& z_i\in\C\\
\end{aligned}\right\}
\end{equation*}
and hence $(\st(\su(2)\times\su(3)),\vt)$ is not \HNR. 
We note however that 
\begin{equation*}
 \wt=\left.\left\{ 
\begin{pmatrix}
 \lambdaup_1&\zetaup_1\\
 \zetaup_2&\lambdaup_2\\
 &&  \lambdaup_1&\zetaup_1& z_1\\
 && \zetaup_2&\lambdaup_2 & z_2\\
 && 0 & 0 &\lambdaup_3 \end{pmatrix}
 \right| 
\begin{aligned}
 & \lambdaup_i,\zetaup_i,z_i\in\C\\
 & 2\lambdaup_1+2\lambdaup_2+\lambdaup_3=0
\end{aligned}
\right\} \subset\vt+\overline{\vt}
\end{equation*}
has a horocyclic $\wt_n.$ 
The orthogonal ${\mt}_0$ of ${\vt}+\overline{\vt}$ in $\st(\pt(2)\times\pt(3))$
is 
\begin{equation*}
{\mt}_0=\left.\left\{ 
\begin{pmatrix}
 X \\
 & -X \\
 && 0
\end{pmatrix}\right| X\in\pt(2)\right\}
\end{equation*}
and, according to Theorem~\ref{thm3.14} %
it can be used to describe the typical fiber of
the Mostow fibration $M_-\to{M}_0$ in this case.
\end{exam}
\subsection{Decomposition of unipotent Lie groups}\label{subnilpo}
A \textit{unipotent} 
Lie group is a connected and simply connected Lie group 
$\Nf$ having a nilpotent
Lie algebra $\nt.$ 
Then the exponential map 
$\exp:\nt\to\Nf$ is an algebraic diffeomorphism and 
each Lie subalgebra $\eg$ of $\nt$ is the Lie algebra of an 
analytic closed subgroup $\Ef$ of~$\Nf$.

\begin{prop}\label{prop:3.5.6}
 Let $\Nf$ be a unipotent Lie group
 and $\Sb$ a group of automorphisms of its Lie algebra
 $\nt,$ which acts on $\nt$ in a completely reducible way.
 If  $\Ef$ a Lie 
 subgroup of $\Nf$ with an $\Sb$-invariant 
 Lie algebra $\eg,$ then we can find an $\Sb$-invariant 
  linear complement
 $\lt$  of $\eg$ in $\nt$ such that 
\begin{equation}\label{eq:3.62}
 \lt\times\Ef\ni (X,x)\longrightarrow \exp(X)\cdot{x}\in\Nf
\end{equation}
is a diffeomorphism onto. 
\end{prop} 
\begin{proof}
 We argue by recurrence on the 
sum of the dimension $n$ of $\nt$ and  the  
 codimension $k$ of $\eg$ in $\nt$. The statement is indeed trivial when $n=1$, or $k=0.$ 
 If $k=1$, then $\eg$ is an ideal in $\nt$ and has a $1$-dimensional $\Sb$-invariant complement
 $\lt$ in $\nt.$ Since $\lt$ is a Lie subalgebra, using e.g. 
 \cite[Lemma 3.18.5]{Varad} we conclude that \eqref{eq:3.62} is a diffeomorphism in this case. 
 \par 
Assume now that $k>1$ and that the statement has already been proved 
for subalgebras $\eg$ of codimension lesser than $k$ or 
nilpotent Lie algebras $\nt$ of
dimension lesser than $n$. Since $\nt$ is nilpotent, its center $\cg$ has positive dimension
and is \mbox{$\Sb$-invariant}. If $\cg\cap\eg\neq\{0\}$, then $\Af=\exp(\cg\cap\eg)$ is a nontrivial
normal subgroup of~$\Nf.$ 
Since $\dim(\Nf/\Af)<n$ and $\Sb$ acts in a completely reducible way on $\nt/(\cg\cap\eg),$ 
by the recursive assumption 
we can find an $\Sb$-invariant  
linear complement $\lt$ of $\et$ in $\nt$ such that, for its projection $\lt'$
in $\nt/(\cg\cap\eg),$ the map
\begin{equation*}
f': \lt' \times (\Ef/\Af) \ni (X',x') \longrightarrow \exp(X')\cdot x' \in \Nf/\Af
\end{equation*}
is a diffeomorphism.
This implies that \eqref{eq:3.62} is also a diffeomorphism.
In fact, if $\zetaup\in\Nf$, by the surjectivity of $f'$ there is a pair $(X,y)\in\lt\times\Ef$
such that $\exp(X)\cdot{y}=\zetaup\cdot a,$ for some $a\in\Af$. This shows that 
$\zetaup=\exp(X)\cdot (y\cdot{a}^{-1})$ and therefore \eqref{eq:3.62} is onto. 
If $\zetaup=\exp(X_1)\cdot (x_1)=\exp(X_2)\cdot(x_2)\cdot{a}$, with $X_1,X_2\in\lt$,
$x_1,x_2\in\Ef$ and $a\in\Af$, then $X_1=X_2=X$ because the projection $\lt\to\lt'$
is a linear isomorphism. Moreover, the correspondence $\zetaup\to{X}$ is $\cC^\infty$-smooth,
because ${f'}^{-1}$ is smooth. Then $\zetaup\to{x}=\exp(-X)\cdot\zetaup\in\Ef$ is also smooth,
and $\zetaup\to(X,\exp(-X)\zetaup)$ yields a smooth inverse of \eqref{eq:3.62}. \par 
If $\cg\cap\eg=\{0\}$, then by the recurrence assumption, 
we can take an $\Sb$-invariant 
 linear complement $\lt$ of $\eg$ in $\nt$
containing $\cg$ and such that
\begin{equation*}
f': (\lt/\cg)\times((\Ef\cdot\Cf)/\Cf)\ni (X',x')\longrightarrow \exp(X')\cdot{x'}\in\Nf/\Cf.
\end{equation*}
is a diffeomorphism. We claim that,
with this choice, 
\eqref{eq:3.62} is a diffeomorphism. Indeed, $(\Ef\cdot\Cf)/\Cf\simeq\Ef$ and therefore
 for $\zetaup\in\Nf$ there is a unique $x\in\Ef$, with $x=\phiup(\zetaup)$ for a smooth
 function $\phiup:\Nf\to\Ef$, such that, 
for some $Z\in\cg$ and $Y\in\lt$, 
\begin{equation*}
 \zetaup\cdot\exp(Z)=\exp(Y)\cdot{x}\Rightarrow \zetaup=\exp(Y-Z)\cdot{x}.
\end{equation*}
The exponential is a diffeomorphism of $\nt$ onto $\Nf$.
If we denote by $\log:\Nf\to\nt$ its inverse, we obtain $X=Y-Z=\log(\zetaup\cdot{x}^{-1})\in\lt$
and $$\Nf\ni\zetaup\to
 \left(\log(\zetaup\cdot[\phiup(\zetaup)]^{-1}),\phiup(\zetaup)\right)\in\lt\times\Ef$$
is a smooth inverse of \eqref{eq:3.62}.
This completes the proof.
\end{proof}
With the notation of the previous section, 
we will apply Proposition~\ref{prop:3.5.6} to the case where $\Nf=\Qf_n$ and
$\nt=\qt_n,$ for a minimal $\qt\in\Pp_0(\wt),$ while $\eg=\vt_n$ and 
$\Sb=\Ad(\Vf_0).$ Since $\Vf_0$ is compact, its
adjoint action on $\qt_n$ is completely reducible.

\subsection{Structure of the typical fiber} 
The quotient $\Kf/\Qf$ of $\Kf$ by a parabolic subgroup $\Qf$ is compact
and thus a homogeneous space of its compact form $\Kf_0.$ Thus
 \begin{equation}\label{forsette}
 \Kf=\Kf_0\cdot \Qf.
\end{equation}
Set $\kt=\Lie(\Kf),$ $\qt=\Lie(\Qf),$ and choose $\Kf_0$ to contain
a maximal compact subgroup of $\Qf$. Then $\Qf$ has a Levi-Chevalley
decomposition $\Qf=\Lf(\Qf)\cdot\Qf_n,$ whose reductive factor
$\Lf(\Qf)$ has Lie algebra $\Li(\qt)=\qt\cap\overline{\qt}.$ 
The conjugation is taken with respect to the real compact form
$\kt_0$ and
$\Qf_n$ is the unipotent factor of $\Qf,$ with Lie algebra $\qt_n.$
We consider the Cartan decomposition $\kt=\kt_0\oplus\po,$ with 
 $\po=i\cdot\kt_0.$ Using the Cartan decomposition of
$\Lf(\Qf),$ we obtain the direct product decomposition
\begin{equation}\label{eq:3.59}
 \Qf=\Lf(\Qf)\cdot\exp(\nt(\qt))=\Lf_0(\Qf)\cdot
 \exp(\pt _0\cap\qt)\cdot \exp(\nt(\qt)).
\end{equation}
\par 
We keep the notation of the previous sections, with $\wt$ the maximal
complex Lie subalgebra with $\vt\subseteq\wt\subseteq\vt+\overline{\vt}$
and take $\qt$  in $\Pp_0(\wt).$  
Then $\et=\vt+\qt_n$ is a Lie subalgebra of $\kt$ and the pair
$(\kt_0,\et)$ has the {\HNR} property. Set 
\begin{equation}\label{fornovo}
\ft_0=\po\cap(\vt+\qt_n)^\perp.
\end{equation}
By \eqref{eqtreventuno}, we obtain the direct product decomposition 
\begin{equation}\label{quattrododici}
\Kf=\Kf_0\cdot \exp(\ft_0)\cdot \exp(\vt_n+\qt_n)\cdot\exp(\vt\cap\po).
\end{equation}
We use Proposition~\ref{prop:3.5.6} to decompose 
$\exp(\vt_n+\qt_n):$ we can find
an $\Ad(\Vf_0)$-invariant 
linear subspace $\lt$  of $(\vt_n+\qt_n)$ such that 
$\vt_n+\qt_n=\lt\oplus\vt_n$ and
\begin{equation}\label{forundici}
\lt\oplus\vt_n\ni (X,Y)\longrightarrow \exp(X)\cdot\exp(Y)\in
\Vf_n\cdot\Qf_n=
\exp(\vt_n+\qt_n)
\end{equation}
is a diffeomorphism. We obtained: 
\begin{thm} \label{tmforotto}
Let $\ft_0$ and $\lt$ be defined by \eqref{fornovo} and
\eqref{forundici}. Then we have a direct product decomposition 
\begin{equation}
\Kf=\Kf_0\cdot\exp(\ft_0)\cdot\exp(\lt)\cdot\Vf',
\end{equation}
where $\Vf'=\exp(\vt_n)\cdot\exp(\vt\cap\po).$\par 
Then $\Ff_0=\exp(\ft_0)\cdot\exp(\lt),$ 
with  the adjoint action of $\Vf_0,$ 
is the typical fiber of
the Mostow fibration: 
\begin{equation}\vspace{-22pt}
M_-\simeq\Kf/\Vf \simeq \Kf_0\times_{\Vf_0}\Ff_0.
\end{equation}
\end{thm} 
\qed
\begin{lem} \label{lem:3.5.11}
If
 $\Nf$ is a unipotent subgoup of $\Kf,$ then, for every $p\in\Pf_0(n)$, the map 
\begin{equation}\label{herm4.10}
 \Nf\ni{z} \longrightarrow z^* p z \in {N}_p=\{z^*pz\mid z\in\Nf\} \end{equation}
is a diffeomorphism.
\end{lem} 
\begin{proof} 
In fact the stabilizer 
$\mathbf{Stab}(p)$ 
of $p$ for the right action $$\Kf\times\Pf_0(\kt)\ni (z,x)\to{z}^*\cdot{x}\cdot{z}\in\Pf_0(\kt)$$
of $\Kf$ on $\Pf_0(\kt)$  
is a compact group and hence has trivial intersection with $\Nf$.
Thus \eqref{herm4.10} %
is a diffeomorphism with the image, being the restriction %
to $
\Nf\simeq\Nf/\{e_{\Kf}\}$ %
of the diffeomorphism 
$\Kf/\mathbf{Stab}(p)\to\Pf_0(\kt)$.
\end{proof}
\begin{cor} \label{cor:3.5.12}
Fix $\qt\in\Pp_0(\wt)$ and let 
$\ft_0$ and $\lt$ be the corresponding subspaces of
$\kt$ of Theorem~\ref{tmforotto}. 
Then the elements $X\in\ft_0$ and $Z\in\lt$ of the
decomposition 
\begin{equation*}
 \zetaup=u\cdot\exp(X)\cdot\exp(Z)\cdot{v},\quad\text{with \;\; $u\in\Kf_0$, $v\in\exp(\vt_n)\cdot\exp(v\cap\pt_0)$}
\end{equation*}
are obtained in the following way: 
\begin{itemize}
\item[(a)]
 $[0,1]\ni{t}\to\exp(2tX)$ is the geodesic in $\Pf_0(\kt)$ joining $e_{\Kf}$ to the unique point 
 $p_0$ of $\tilde{N}_{\zetaup^*\cdot\zetaup}
 =\{z^*\cdot\zetaup^*\cdot\zetaup\cdot{z}\mid z\in\Vf\cdot\Qf_n\}$
at minimal distance from $e_{\Kf};$
\item[(b)] $Z$ is the unique element of $\lt$ such that \; $\exp(Z^*)\cdot{p}_0\cdot\exp(Z)$ belongs to 
$N_{p_0}=\{z^*\cdot{p}_0\cdot{z}\mid z\in\Vf\}$.
\end{itemize}
 \end{cor} 
\begin{proof}
 Indeed the Mostow fibration of $M_-'=\Kf/(\Vf\cdot\Qf_n)$ can be taken to have a \textit{hermitian}
 typical fiber $\exp(\ft_0)$ and correspondingly we obtain a unique decomposition
\begin{equation*}
 \zetaup= u \cdot \exp(X)\cdot\xiup\cdot\exp(Y) \;\;\text{with\;\; $\xiup\in\Qf_n$ and $Y\in\vt\cap\pt_0,$}
\end{equation*}
The characterization of $X$ coming from the proof of Proposition~\ref{propo3.3.4}
yields (a). \par
Next we consider \; $p_{\xiup}=\xiup^*\cdot\exp(2X)\cdot\xiup=\xiup^*\cdot{p_0}\cdot\xiup$.
By Lemma~\ref{lem:3.5.11} and the choice of $\lt$ 
we know that the element $p_{\xiup}$
of 
$\{z^*\cdot{p_0}\cdot{z}\mid z\in\Qf_n\}$
uniquely decomposes as a product ${w^*}\cdot\exp(Z^*)\cdot{p}_0\cdot\exp(Z)\cdot{w}$
with $w\in\Vf_{\!{n}}$ and $Z\in\lt$. This completes the~proof.
\end{proof}
\section{Application to Dolbeault and \texorpdfstring{$CR$}{TEXT} cohomologies}\label{cohomo}
The cohomology groups of the tangential Cauchy-Riemann complex on real-analytic forms
on $M_0$ is the inductive limit of the corresponding Dolbeault cohomology groups of its
tubular neighborhoods in $M_-.$ We know  by \cite{HN1995} that 
in some degrees
these groups coincide with those computed on tangential
smooth forms or on currents.
We will employ Andreotti-Grauert theory to compare the tangential $CR$ 
cohomology  on $M_0$
with the corresponding \textit{global} Dolbeault cohomology of $M_-.$
To this aim 
we will use the Mostow fibration $M_-\to{M}_0$ to construct
 a non negative exhaustion fuction for $M_-$,  vanishing on $M_0,$ and
 having a complex Hessian whose signature reflects
the pseudoconvexity/pseudoconcavity of $M_0.$
In this way we prove relations of
 the $CR$ cohomology of $M_0$ with the Dolbeault cohomololy of the 
 $\Kf$-orbit  $M_-$, similar to what J.A.Wolf did  in \cite{WoSc84}
for the relationship of 
the open orbits $M_+$ of a real form $\Gf_0$ 
of a complex semisimple Lie group $\Gf$
in a flag $M$ of $\Gf$ with the structure 
of their Matsuki duals $M_-=M_0,$ which in this case are 
compact complex manifolds.
%
\subsection{An Exhaustion Function for $M_{-}$} \label{k0orb}
In \cite{gr58} H. Grauert  noticed
that a real-analytic 
manifold admits  a 
fundamental systems of Stein tubular
neighborhoods in any of its complexifications. 
In fact,  a homogeneous analogue of Grauert's theorem 
is the fact that 
the complexification $\Kf$ 
of a compact Lie group $\mathbf{K}_0$ 
is Stein, and the isomorphism provided by the Cartan decomposition 
\begin{equation*}
 \mathbf{K}_0\times\mathfrak{k}_0\ni (x,X) \longrightarrow x\cdot\exp(iX)\in\mathbf{K}
\end{equation*}
also yields the exhaustion function 
\begin{equation*}
 \mathbf{K}\ni  x\cdot \exp(iX)\longrightarrow \|X\|^2=-\kll(X,X)\in\mathbb{R},
\end{equation*}
which is zero on $\mathbf{K}_0$,
 positive on $\mathbf{K}\setminus\mathbf{K}_0$ and  strictly pseudo-convex everywhere.
Here and in the following we shall denote by $\kll$ 
both 
 the negative definite invariant form of a  faithfull unitary
representation of $\mathfrak{k}_0$ 
and its $\C$-bilinear extension to~$\kt$. When $\kt_0$ is semisimple, 
the adjoint representation is faithful and we may take as $\kll$ 
the 
Killing form.\par 
We 
proceed in a similar way to construct an exhaustion function on $M_-$ for the
canonical embedding $M_0\hookrightarrow{M}_-$ of a $\nt$-reductive 
$\Kf_0$-homogeneous compact $CR$ manifold $M_0$.
We use the notation of the previous sections.
\par
Assume that the pair $(\kt_0,\vt)$ is $\nt$-reductive and {\HNR}.
We already noticed that the last condition is \textit{natural} if we consider on $M_0$
\textit{maximal} $\Kf_0$-invariant $CR$ structures.
Then, by Corollary~\ref{cor:3.5.12}, we have a direct product decomposition 
\begin{gather}\label{eq:4.1}
 \Kf=\Kf_0\cdot\exp(\ft_0)\cdot\exp(\vt_n)\cdot\exp(\vt\cap\pt_0)
\intertext{with \; $\pt_0={\mathpzc{i}}\!\cdot\!\kt_0$\; and \;$\ft_0=(\vt+\overline{\vt})^\perp\cap\pt_0.$
Moreover,
the $\exp(\ft_0)$-term in \eqref{eq:4.1} is characterized by} 
\label{eq:4.2}
\left\{ \begin{gathered}
\text{if}\;\; 
\zetaup=u\cdot\exp(X)\cdot{v},\;\;\text{with\;\; $u\in\Kf_0,$ $X\in\ft_0$ and $v\in\Vf$, then} \\
\|X\|= \tfrac{1}{2}\dist(\zetaup^*\zetaup,N),\;\; \text{for}\;\;
N=\{v^*\cdot v\mid v\in\Vf\}
\subset\Pf_0(\kt).
\end{gathered}
\right.
\end{gather}
This is indeed a consequence of  
Corollary~\ref{cor:3.5.12} when $\lt=\{0\}.$
\par 

By passing to the quotient, 
 the map 
\begin{equation*}
 \mathbf{K}_0\times\mathfrak{f}_0\ni (x ,X) \longrightarrow \|X\|^2=\kll(X,X)\in\mathbb{R}.
\end{equation*}
defines a smooth exhaustion function (as usual
square brackets mean equivalence classes)
\begin{equation}\label{exha}
 \phiup:M_-\simeq \Kf_0\times_{\If_0}\ft_0\ni [x,X]\longrightarrow \|X\|^2\in\R.
\end{equation}
We have: 
\begin{lem}\label{exhalemm}
If $(\kt_0,\vt)$ is $\nt$-reductive and 
{\HNR}, 
 then the map $\phiup$ of \eqref{exha} has the properties:  
\begin{enumerate}
\item $\phiup\in\Ci(M_-,\R)$ and $\phiup\geq{0}$ on $M_-$;
\item $\phiup^{-1}(0)=M_0$ 
and $d\phiup\neq{0}$ if $\phiup>0$ 
; 
\item $\phiup$ is invariant under the left action of $\mathbf{K}_0$ on $M_-$ : 
\begin{equation*}\vspace{-20pt}
 \phiup({x}\cdot{ p})=\phiup(p),\quad\forall p\in M_{-}\,,\;\;\forall {x}\in\mathbf{K}_0.
\end{equation*}
\qed 
\end{enumerate} 
\end{lem}

\begin{ntz}\label{Uc}The \emph{level 
and sublevel 
sets}\index{Level sets} of  
$\phiup$ 
will be 
 denoted by 
\begin{equation} 
\Phi_c=\{p\in M_-\mid \phiup(p)=c\}\Subset M_-\;
\;\; \text{and}\;\;\; 
\Omega_c=\{p\in{M}_-\mid \phiup(p)<c\}.
\end{equation}
\end
{ntz}

\subsection{$\mathbf{K}_0$-Orbits in $M_-$}\label{k1orb}
The level sets $\Phi_c$ are foliated by $\Kf_0$-orbits.
Since all points of $M_-$ have 
representatives  of the form $x\cdot\exp(X)$  with 
 $x\in\Kf_0$ and $X\in\ft_0$, then every $\mathbf{K}_0$-orbit  
intersects the fiber $F_0$  
over the 
base point $p_0$ at a point $p_X=[\exp(X)],$ 
for some $X\in\ft_0$. An $x\in\Kf_0$ stabilizes $p_X$ if and only if
$x\cdot\exp(X)$ is still 
a representative of $p_X$, and this, by the equivalence relation defining
$\Kf_0\times_{\If_0}\ft_0$,  means that $x\in\If_0$ and $\Ad(x)(X)=X$. 
Indeed the equation $x\exp(X)z=\exp(X)$ with $z\in\Vf$ implies, by the uniqueness of the
Mostow decomposition, that $z=x^{-1}\in\If_0$ and $x\exp(X)x^{-1}=\exp(\Ad(x)(X))=\exp(X),$
yielding $\Ad(x)(X)=X$.
\par 
 Thus the  $\mathbf{K}_0$-orbit 
 \begin{align}\label{Mt}
 M_X&=\{x\cdot{p_X}=[x\cdot\exp(X)]\mid x\in\Kf_0\}
\intertext{in $M_-$ through $p_X$ can be identified with 
the homogeneous space 
 $\mathbf{K}_0/{\Vf}_X$, where}
\notag \Vf_X&=
\{x\in\If_0\mid \Ad(x)(X)=X\},
\intertext{is the stabilizer  of $p_X$ in $\Kf_0.$ It
is a closed Lie subgroup of $\Kf_0$ 
 with Lie algebra}\notag
\vt_X&=\{Y\in\vt_0\mid [Y,X]=0\}.
\end{align}
\begin{lem}\label{lemma5.2}
 $M_X$ is a compact $\Kf_0$-homogeneous $CR$-manifold with $CR$-algebra
 $\big(\kt_0,\Ad(\exp(X))(\vt)\big)$ at $p_X=[\exp(X)].$\qed
\end{lem}
\begin{rmk}\label{prop.cr.diff}
In general, $M_X$ may not be diffeomorphic to $M_0$.
Indeed, 
$M_0$ is a \textit{minimal} $\Kf_0$-orbit in $M_-$ and  
$M_X$ is diffeomorphic (and $CR$-diffeomorphic) to $M_0$ 
if and only if $M_X$ and $M_0$ have the same dimension.
\end{rmk}
\smallskip
\par
For $X\in\ft_0$, the left translation 
$
M_-\ni{p}\longrightarrow \exp(X)\cdot{p}\in{M}_-
$
is a biholomorphism of $M_-$ which transforms $M_0$ onto a $CR$-diffeomorphic 
submanifold 
\begin{equation}\label{tildeMx}
\tilde{M}_X=\exp(X)\cdot{M}_0.\end{equation}

\begin{lem}\label{containedlemma} For $X\in\ft_0$, we have 
\begin{equation}
 \tilde{M}_X\subset\{\phiup\leq\|X\|^2\}
 =\overline{\Omega}_{\|X\|^2} .
\end{equation}
 \end{lem} 
\begin{proof}
 Let $\pi:\Kf\ni\zetaup
 \to[\zetaup]\in\Kf/\Vf\simeq{M}_-$ be the canonical projection. Any point of $M_0$ is 
 $\pi(u)$ for some $u\in\Kf_0$ and  
then the points $p$ of $\tilde{M}_X$ have the form 
  $p=\exp(X)\pi(u)=\pi(\exp(X)\cdot{u})$. Set $\zetaup=\exp(X)\cdot{u}$.
We know that $\phiup(p)$ is the square of  the half-distance 
in $\Pf_0(\Kf)$ from the base point
$e_{\Kf}$ to
$$N_{\zetaup^*\zetaup}=\{v^*\!\cdot\! \zetaup^*\!\cdot\! \zetaup\!\cdot\! {v}\mid v\in\Vf\}.$$ 
Since the point
$(\zetaup^*\!\cdot\!\zetaup)$ belongs to $N_{\zetaup^*\cdot\zetaup}$ and has distance $2\|X\|$ from
$e_{\Kf}$, (in fact ${t\to{u}^*\!\cdot\!\exp(2tX)\!\cdot\!{u}}$ is the geodesic joining $e_{\Kf}$ to 
$({\zetaup^*\!\cdot\!\zetaup})$), it follows that \mbox{$\phiup(p)\leq\|X\|^2$.}
\end{proof}
We summarize: 
\begin{prop}
 Let $c>0$. Then 
\begin{equation}\label{Ucc}\vspace{-27pt}
 \Phi_c={\bigcup}_{ 
\begin{smallmatrix}
 X\in\ft_0,\\
 \|X\|^2=c
\end{smallmatrix}} M_X\;\; \text{(disjoint union)},\quad 
\tilde{M}_X\subset\{\phiup\leq\|X\|^2\},\;\; \forall X\in\ft_0.
\end{equation}
\qed
\end{prop}
\smallskip
In particuar, for $c>0$,  we can draw through each point of  $\Phi_c$
a translate $\tilde{M}_X$ of $M_0,$ which is
$CR$-diffeomorphic to $M_0$ and tangent to $\Phi_c$ \textit{from inside}, i.e.
lying in $\overline{\Omega}_c.$
This means that 
the boundary 
$U_c$  
of $\Omega_c$ is at each point \textit{less convex} than $M_0.$
\subsection{Application to Dolbeault and 
\texorpdfstring{$CR$}{TEXT} Cohomologies I}\label{vanishing.sec}
By  An\-dreotti-Grauert theory (see \cite{AG62}) 
we know that for every coherent sheaf $\mathcal{F}$ 
on  an $r$-pseudoncave
complex  manifold $\mathrm{X}$ we have 
\begin{equation*}
\mathbf{H}^{j}(\mathrm{X},\mathcal{F})<\infty,\;\;\;\forall <r-\mathrm{hd}(\mathcal{F}), 
\end{equation*}
where $\mathrm{hd}(\mathcal{F})$ is the homological dimension
of $\mathcal{F}.$ \par 
We obtain the following:
\begin{thm}\label{vanishing}
Let $M_0$ be a compact $\nt$-reductive 
homogeneous $CR$ manifold, with $(\kt_0,\vt)$
 {\HNR} and canonical complex
 embedding $M_0\hookrightarrow{M}_-$.\par 
If  $M_0$ is an $r$-psudoconvave $CR$-manifold, then 
 $M_-$ is an $r$-pseudoconcave complex manifold and
for every coherent sheaf $\mathcal{F}$ we have
\begin{equation} \label{vanishingth} 
\dim\big(\mathbf{H}^{j}(M_0,\mathcal{F})
\simeq\mathbf{H}^{j}(M_-,\mathcal{F})\big)
<\infty, 
\;\;\forall j<r-\mathrm{hd}(\mathcal{F}).
\end{equation}
In particular,
\begin{equation}\dim\big(
\mathbf{H}^{p,j}(M_0)\simeq \mathbf{H}^{p,j}(M_-)
\big)<\infty,\;\;\forall j<r.
\end{equation}
\end{thm} 
Here we
used the notation $\mathbf{H}^{p,j}$ for the $\bar{\partial}$
and $\bar{\partial}_{M_0}$-cohomologies on forms of type $(p,*)$. 
Because of the validity of the
Poincar\'e lemma in degree $j$, for 
$0<j<r$
(see \cite{23}), they coincide with the \v{C}ech
cohomology with coefficients in the sheaf of germs of
$CR$ or holomorphic $p$-forms. Moreover, in this range, the tangential Cauchy-Riemann complexes
on currents, $\Ci$-smooth forms and real-analytic forms  on $M_0$
have  isomorphic finite dimensional cohomology groups. 
\begin{proof}
By the \textit{HNR} assumption, the exhaustion function 
$\phiup$ in \eqref{exha}   is   well defined. 
Then to verify  
\eqref{vanishingth} 
we can apply Andreotti-Grauert's theory, after
showing that, for $c>0,$  
each subdomain $\Omega_c=\{\phiup<c\}$ is $r$-pseudoconcave. 
To this aim, 
we prove that the complex Hessian of $\phiup$  admits at least  $r$ negative eigenvalues 
on the analytic 
tangent to $\Phi_c
=\partial\Omega_c.$ By exploiting 
the  $\Kf_0$-invariance of $\phiup$,  
we can, without any loss of 
generality, restrict our consideration to points
$p_0=[\exp{(X)} ] \in \Phi_c,$ with $\| X\|^2=c\in\R$. 
We may consider  $(0,1)$-vector fields which are tangent 
to the submanifold $\tilde{M}_X$, 
defined  in  \eqref{tildeMx} and that are also tangent to $\partial{\Omega}_c$ at
$p_0,$ because
$\tilde{M}_X$ is tangent to $\Phi_c$ 
at $p_0.$
By Lemma~\ref{containedlemma},   
$\tilde{M}_X$
is contained in $\overline{\Omega}_c=\{\phiup\leq\|X\|^2\}.$
 Since $\tilde{M}_X$ is \textit{CR}-diffeomorphic to $M_0$,
 it is $r$-pseudoconcave. Being $\tilde{M}_X\subset\overline{\Omega}_c$,
 the restriction of the complex Hessian of
 $\phiup$ to the analytic tangent to $\tilde{M}_X$ at $p_0$
 has at least as many negative eigenvalues as 
 the Levi form of $\tilde{M}_X$ in the codirection $Jd\phiup([\exp(X)]),$ which,
 by the assumption, are at least $r.$ 
This completes the proof.
\end{proof}
\subsection{Application to Dolbeault and \texorpdfstring{$CR$}{TEXT} cohomologies II}
In this section we want to exploit
the amount of pseudo-convexity of the exhaustion function $\phiup.$
We keep the assumption that $(\kt_0,\vt)$ is $\nt$-reductive and
{\HNR} and set \mbox{$\qt=\{Z\in\kt\mid [Z,\vt_n]\subset\vt_n\}$}
for the maximal parabolic subalgebra in  
$\Pp_0(\vt).$ We recall that $\vt_n=\qt_n$ is the nilradical of $\qt.$ 
Let $\Qf$ be the parabolic subgroup of $\Kf$ with $\Lie(\Qf)=\qt$ and $\Qf_r$ its conjugation-invariant
reductive factor. Let $\varpi:\Kf\to{M}_-=\Kf/\Vf$ be the quotient map. The image of $\Qf_r$ by
$\varpi$ is a $\Qf_r$-homogeneous complex submanifold $Q_-$ of~${M}_-.$ 
\begin{lem} \label{lemma5.7}
For every $X\in\ft_0,$ the $CR$ manifold $\tilde{M}_X$ and the complex manifold
$Q_-$ are transversal at $p_X$ and their analytic tangent spaces at $p_X$ are orthogonal for
the complex Hessian of $\phiup.$
\end{lem} 
\begin{proof}
The pull-backs of $T^{0,1}_{p_X}\tilde{M}_X$ and $T^{0,1}_{p_X}Q_-$ 
to the base point $p_0$ 
by the
bi-holo\-morphic map ${p\to\exp(X)\cdot{p}}$ are, respectively, $\vt_n$ and $\qt_r/(\vt\cap\overline{\vt}).$ 
This is a consequence of the fact that $X\in\qt_r.$
The statement follows from the fact that $\qt_r=\bar{\qt}_r$ and $[\qt_r,\vt_n]\subset\vt_n,$
$[\qt_r,\bar{\vt}_n]\subset\bar{\vt}_n.$
\end{proof}
\begin{thm}\label{vanishingconv}
Let $M_0$ be a compact $\nt$-reductive 
homogeneous $CR$ manifold
of type $(n,k),$ with $(\kt_0,\vt)$
 {\HNR} and canonical complex
 embedding $M_0\hookrightarrow{M}_-$.\par 
If  $M_0$ is an $r$-psudoconvave $CR$-manifold, then 
 $M_-$ is $n-r$-pseudoconvex complex manifold and
for every coherent sheaf $\mathcal{F}$ we have
\begin{equation} \label{vanishingthc} 
\dim\big(\mathbf{H}^{j}(M_0,\mathcal{F})
\simeq\mathbf{H}^{j}(M_-,\mathcal{F})\big)
<\infty, 
\;\;\forall j> n-r .
\end{equation}
In particular,
\begin{equation}\dim\big(
\mathbf{H}^{p,j}(M_0)\simeq \mathbf{H}^{p,j}(M_-)
\big)<\infty,\;\;\forall j>n-r.
\end{equation}
\end{thm} 
\begin{proof}
 By  \cite[Theorem 2.1]{HN1995}, under the $r$-pseudoconcavity assumption, the 
 tangential $CR$ 
 cohomology
 groups on $M_0$ are the inductive limits of the corresponding groups of sheaf and Dolbeault cohomology
 of the tubular neighborhoods of $M_0$ in $M_-.$
 While computing the Levi form of $\phiup$, it suffices to note that its restriction to $Q_-$ is strictly
 pseudo-convex, since it is the exhaustion function associated to the 
 canonical {$CR$-embedding} $M_0\cap{N}_-\hookrightarrow{N}_-$ of a 
 totally real $(\Kf_0\cap\Qf_r)$-homogeneous 
manifold.
Indeed, by \cite[Theorem 4.1]{BiNe69}, the distance from
the totally geodesic submanifold $N'=\{\zetaup^*\zetaup\mid\zetaup\in 
\Vf\cap\Qf_r\}$ in the negatively curved 
space $\Mf'=\Qf_r/(\Qf_r\cap\Kf_0)$ is strictly convex on
$\Mf'\setminus{N}',$ and $\phiup|_{Q_-}$ pulls back on $\Qf_r$
to the composition of $\zetaup\to\zetaup^*\zetaup$ with the square
of the distance from $N'.$ \par  
Hence, for $X\neq{0},$
the complex Hessian of $\phiup$ restricts 
to a Hermitian symmetric form having, by Lemma~\ref{lemma5.7}, at least $r+k-1$ positive eigenvalues
on the analytic tangent of $\Phi_c$ at $p_X.$ 
\par 
 The thesis is then a consequence of the isomorphisms proved
 in \cite[\S{20}]{AG62}.
\end{proof}
\begin{exam} Fix integers $1\leq{p}<q\leq{n}$ and consider
the \textit{real} action of $\SL_{n+1}(\C)$ on the Cartesian product
$\Gr_{\! p}(\C^{n+1})\times\Gr_{\! q}(\C^{n+1})$
of the Grassmannians of $p$ and $q$ planes, 
described by \begin{equation*}
a\cdot (\ell_p,\ell_q)=(a(\ell_p),\bar{a}(\ell_q)),\;\;
\forall \, a\in\SL_{n+1}(\C),\;  \ell_p\in\Gr_{\! p}(\C^{n+1}),\;
 \ell_q\in\Gr_{\! q}(\C^{n+1}) . 
\end{equation*} 
The orbits of the real form $\Gf_0=\SL_{n+1}(\C)$ are parametrized
by the dimension of the intersection ${\ell}_p\cap\overline{\ell}_q:$
with $k_0=\max\{0,p+q-n-1\}$ we have the orbits
\begin{equation*}
M_+(k)=\{(\ell_p,\ell_q)\in \Gr_{\! p}(\C^{n+1})\times\Gr_{\! q}(\C^{n+1})
\mid \dim_{\C}({\ell}_p\cap\overline{\ell}_q)=k\},\;\;
k_0\leq{k}\leq{p}.
\end{equation*}
The complexification $\Kf=\SL_{n+1}(\C)$ of the compact form $\Kf_0=\SU(n+1)$
acts on $\Gr_{\! p}(\C^{n+1})\times\Gr_{\! q}(\C^{n+1})$ by 
\begin{equation*}
a\cdot(\ell_p,\ell_q)=(a(\ell_p),{^T\!{a}^{-1}}(\ell_q)),\;\;
\forall a\in\SL_{n+1}(\C),\;  \ell_p\in\Gr_{\! p}(\C^{n+1}),\;
 \ell_q\in\Gr_{\! q}(\C^{n+1}) .
\end{equation*}
Consider the polarity $\Gr_{\! h}(\C^{n+1})\ni\ell_h\to\ell_h^0
\in\Gr_{\! n+1-h}(\C^{n+1})$ 
 defined by the symmetric bilinear form
\par
\centerline{$b(v,w)=(^T\!{w})\cdot{v}={\sum}_{i=0}^n{v_iw_i}.$}
\par\noindent
Then the orbits of $\Kf$ in 
$\Gr_{\! p}(\C^{n+1})\times\Gr_{\! q}(\C^{n+1})$
are parametrized by: 
\begin{equation*}
M_-(k)=\{(\ell_p,\ell_q)\in \Gr_{\! p}(\C^{n+1})\times\Gr_{\! q}(\C^{n+1})
\mid \dim_{\C}({\ell}_p\cap\ell_q^0)=p-k\},\;\;
k_0\leq{k}\leq{p}.
\end{equation*}
The manifolds $M_+(k)$ and $M_-(k)$ are Matsuki-dual to each other. 
In fact, since 
$\SU(n+1)$ preserves Hermitian orthogonality in $\C^{n+1}$ 
and $\overline{\ell}_q$ and $\ell_q^0$ are Hermitian orthogonal in $\C^{n+1},$ 
the pair $(\ell_p,\ell_q)$ belongs to $M_0(k)=M_+(k)\cap{M}_-(k)$ iff 
\begin{equation*}
\ell_p=(\ell_p\cap\overline{\ell}_q)\oplus (\ell_p\cap\ell_q^0), \;\; \text{and either
$\dim(\ell_p\cap\overline{\ell}_q)=k,$ or $\dim(\ell_p\cap\ell_q^0)=p-k.$}
\end{equation*}
Set $n_1=p-k,$ $n_2=k,$ $n_3=n+1+k-p-q,$ $n_4=q-k.$ Then, taking as base point, with obvious notation, 
$p_0=(\C^{n_1}\oplus\C^{n_2},\C^{n_2}\oplus\C^{n_4}),$
the stabilizer of $p_0$ in $\Kf=\SL_{n+1}(\C)$ has Lie algebra 
\begin{equation*}
 \vt=\left.\left\{ 
\begin{pmatrix}
 Z_{1,1} & Z_{1,2} & Z_{1,3}& Z_{1,4} \\
 0 & Z_{2,2} & 0 & Z_{2,4}\\
 0 & 0 & Z_{3,3} & Z_{3,4}\\
 0 & 0 & 0 & Z_{4,4}.
\end{pmatrix}\right| Z_{i,j}\in\C^{n_i\times{n}_j}\right\}\cap\slt_{n+1}(\C).
\end{equation*}
Indeed, in the block matrix $Z=(Z_{i,j})_{1\leq{i,j}\leq{4}}$ se have 
$Z_{3,1}=0,$ $Z_{3,2}=0,$ $Z_{4,1}=0,$ $Z_{4,2}=0$ because
$Z(\langle e_1,\hdots,e_p\rangle)\subset \langle e_1,\hdots,e_p\rangle.$
Moreover, the inclusion  \par\noindent 
$\trasp{Z}(\C^{n_2}\oplus\C^{n_4})\subset\C^%
{n_2}\oplus\C^{n_4}$
is equivalent to 
\begin{equation*}
\begin{pmatrix}
\trasp{Z_{1,1}} &\trasp{Z_{2,1}}&0&0\\
\trasp{Z_{1,2}} &\trasp{Z_{2,2}}&0&0\\
\trasp{Z_{1,3}} &\trasp{Z_{2,3}}&\trasp{Z_{3,3}}&\trasp{Z_{4,3}}\\
\trasp{Z_{1,4}}& \trasp{Z_{2,4}}&\trasp{Z_{3,4}}&\trasp{Z_{4,4}}
\end{pmatrix} \begin{pmatrix}
0\\ X_2 \\ 0 \\ X_4
\end{pmatrix}= \begin{pmatrix} 0 \\ Y_2 \\ 0 \\ Y_4\end{pmatrix},\;\;
\forall X_2\in\C^{n_2},\; X_4\in\C^{n_4},
\end{equation*}
and this yields  $Z_{2,1}=0,$ $Z_{2,3}=0,$ $Z_{4,3}=0.$ 
The compact $CR$ manifold $M_0(k)$ has $CR$ dimension equal to 
$\nuup=(n_1n_2+n_1n_3+n_1n_4+n_2n_4+n_3n_4)$
and $CR$-codimension $d=2n_2n_3.$  
The case $k=k_0,$ where $n_3=0,$ 
is the one where $\vt$ is parabolic, and $M_0(k_0)=M_-(k_0)$ is a
complex flag manifold.
In general, 
$(\kt_0,\vt_n)$ is {\HNR} 
because 
\begin{equation*}
 \vt_n=\left.\left\{ 
\begin{pmatrix}
 0& Z_{1,2} & Z_{1,3}& Z_{1,4} \\
 0 & 0& 0 & Z_{2,4}\\
 0 & 0 & 0 & Z_{3,4}\\
 0 & 0 & 0 & 0
\end{pmatrix}\right| Z_{i,j}\in\C^{n_i\times{n}_j}\right\}\cap\slt_{n+1}(\C)
\end{equation*}
is the nilpotent radical of \begin{equation*}
\qt=\left.\left\{ 
\begin{pmatrix}
 Z_{1,1}& Z_{1,2} & Z_{1,3}& Z_{1,4} \\
 0 & Z_{2,2}& Z_{2,3} & Z_{2,4}\\
 0 & Z_{3,2} & Z_{3,3} & Z_{3,4}\\
 0 & 0 & 0 & Z_{3,4}
\end{pmatrix}\right| Z_{i,j}\in\C^{n_i\times{n}_j}\right\}\cap\slt_{n+1}(\C)
\end{equation*}
Then \begin{equation}
\ft_0=\mt_0=\left.\left\{ \begin{pmatrix}
0 & 0 & 0 & 0\\
0 & 0 & Z_{2,3} & 0 \\
0& -Z_{2,3}^* & 0 & 0\\
0 & 0 & 0 & 0\end{pmatrix}
\right| Z_{2,3}\in\C^{n_2\times{n_3}}\right\}\simeq\C^{n_2\times{n_3}} .
\end{equation}
The $CR$ algebra $(\kt_0,\vt)$ is weakly degenerate when 
$k<p$ and strictly nondegenerate, according to \cite{MN05}, when $k=p.$ 
The \textit{vector valued Levi form} is 
\begin{equation*}
(Z_{1,2},Z_{1,3},Z_{1,4},Z_{2,4},Z_{3,4}) \to 
Z_{1,2}^*Z_{1,3}+Z_{2,4}Z_{3,4}^*
\end{equation*}
and hence all the nonzero \textit{scalar Levi form} have 
a Witt index equal to  
$\muup\! =\! 
{(n_1\! +\! n_4)}=(p\! -\! k)\!+\! (q\! -\! k)=p+q-2k.$ 
The complex manifold $M_-(k)$ has dimension 
$N=n_1n_2+n_1n_3+n_1n_4+n_2n_3
+n_2n_4+n_3n_4$ and, according to 
Theorems~\ref{vanishing} and~\ref{vanishingconv}
is $\muup$-pseudoconcave and $(\nuup-\muup)$-pseudoconvex.
\end{exam}
\subsection{Application to Dolbeault and \texorpdfstring{$CR$}{TEXT}
cohomologies III} In this section we extend Theorem~\ref{vanishing}
to the case where we do not assume that $(\kt_0,\vt)$ is {\HNR}. 
To this aim we utilize an $r$-pseudoconcave exhausting functions 
which is only continuous (see \cite{Bun1990, Dieu2006, HuMu1978, SL1984}).
Namely, we will consider the function 
\begin{equation} \label{defphi}
\phiup([\zetaup])=\dist^2(\zetaup^*\zetaup,N
),\;\; \text{for
$\zetaup\in\Kf,$}
\end{equation}
where $N
=\{v^*v\mid v\in\Vf\}$ as in \eqref{defenne} and $[\zetaup]$ is the element
of $M_-=\Kf/\Vf$ corresponding to $\zetaup\in\Kf.$ \par 
We recall that a continuous function $\phiup,$ defined on a complex
$\nuup$-dimensional 
manifold $M_-$, is said to be \emph{weakly 
$r$-pseudoconcave} if, for every
point $p\in{M}_-,$ we can find a coordinate neighborhood $(U,z),$ centered
at $p,$ such that,
for every $(\nuup-r+1)$-dimensional linear subspace $\ell$ 
of $\C^\nuup,$ for every coordinate ball $B\Subset{U}$ and
$\psiup$ plurisubharmonic on a neighborhood of $\bar{B},$ 
with $\phiup\geq\psiup$ on $\ell\cap\partial{B}$ we also have
$\phiup\geq\psiup$ on $\ell \cap{B}.$  \par 
We say that $\phiup$ is \emph{strictly $r$-pseudoconcave} if, for each
$p\in{M}_-,$ we can find an open 
coordinate neighborhood $(U,p)$ centered in $p$  and
an $\epsilon>0$ such that 
$\phiup+\epsilon |z|^2$ is weakly
$r$-pseudoconcave in $U.$ \par 
By Bungart's approximation theorem (\cite[Theorem 5.2]{Bun1990})
strictly $r$-pseu\-do\-con\-ca\-ve functions can be uniformly approximated 
on compacts 
by piece-wise smooth strictly $r$-pseudoconcave functions. 
Thus (see e.g.  \cite[Chapter IV]{Andreotti1973}) we can still apply the 
Andreotti-Grauert theory when we have a 
strictly-$r$-pseu\-do\-con\-ca\-ve exhaustion function
which is only continuous.\par 
Our application relies then on the following lemmas.
\begin{lem} Let $\phiup$ be a continuous exhaustion function on
$M_-$ and assume that, for all $c>0$ and $p_0\in\Phi_c=\{p\in{M}_-\mid
\phiup(p)=c\}$ there is a germ of $CR$ generic $r$-pseudoconcave $CR$ 
submanifold $M_0(p_0)$ of $M_-$ through $p_0$ with
 $M_0(p_0)\subset\{\phiup_p\leq{c}\}.$ Then $\phiup$ is
weakly $r$-pseudoconcave. \begin{proof}
 We argue by contradiction, assuming that, for every coordinate
 neighborhood $(U,z)$ centered at a point $p_0\in{M}_,$ 
 we can find a
 $(\nuup-r+1)$-dimensional linear subspace $\ell$ of $\C^{\nuup}$
 and a plurisubharmonic $\psiup,$ defined on a neighborhood 
 of the closure $\bar{B}$ of a coordinate ball in $U$, and a point
 $p_1\in{\ell}\cap B$ where $\phiup(p_1)<\psiup(p_1),$ while
 $\phiup(p)\geq\psiup(p)$ for all $p\in\partial{B}\cap\ell.$ 
 Clearly the same condition is satisfied by any
 linear $(\nuup-r+1)$-plane sufficiently close to $\ell,$ so that
 we can assume that $\ell$ intersects $M_0(p_1)$ transversally.
The intersection $M_0(p_1)\cap\ell$ is then a $1$-pseudoconcave
$CR$ submanifold of $\ell$, but the restriction of $\psiup$ to
$\ell\cap{M}_0(p_1)\cap\bar{B}$ contradicts then the maximum principle,
since takes at the interior point $p_1$ a value larger than
the supremum of the values taken
on the boundary $\ell\cap{M}_0(p_1)\cap\bar{B}$ (see e.g.
\cite{HN2003}). The
contradiction  proves that $\phiup$ is weakly $r$-pseudoconcave.
 \end{proof}
\end{lem}
\begin{lem} The exhaustion function $\phiup$ defined by \eqref{defphi}
is strictly $r$-pseu\-do\-con\-ca\-ve on $M_-\setminus{M}_0.$ 
\end{lem} 
\begin{proof} 
By Proposition~\ref{proptreotto}, there is $c_0>0$ such that
$\phiup$ is 
strictly $r$-pseudoconcave when 
$0<\phiup(p)\leq{c_0^2}$, since, by \cite[Lemma 2.6]{HuMu1978},
for a smooth function the notion of strict $r$-pseudoconcavity coincides with the requirement about
the signature of its complex Hessian.
\par 
For $\zetaup\in\Kf,$ we can consider the function $\phiup_{\zetaup}(p)=\phiup(\zetaup^{-1}\cdot{p}),$
which is continuous and 
weakly $r$-pseudoconcave on $M_-\setminus(\zetaup\cdot{M}_0)$ and strictly
$r$-pseudoconcave when it takes positive values smaller than $c_0^2.$ 
Let $p_0\in{M}_-$ with $\phiup(p_0)>c_0^2$ and fix a relatively compact coordinate neighborhood
$(U,z)$ in $M_-,$ centered at $p_0.$ We can assume, for a fixed $\deltaup$ with
$0<2\deltaup<c_0,$ that \mbox{$U\subset\{p\mid |\phiup(p)-\phiup({p_0})|<\deltaup^2\}.$} 
We observe that $\phiup(p)=\inf_{\phiup([\zetaup])=\phiup(p_0)-\deltaup} (\sqrt{\deltaup}+
\sqrt{\phiup_{\zetaup}(p)})^2.$ The functions $p\to\etaup_{\zetaup}(p)=(\sqrt{\deltaup}+
\sqrt{\phiup_{\zetaup}(p)})^2, $ when $\phiup(\zetaup)=\phiup(p_0)-\deltaup^2,$
are uniformly strictly $r$-pseudoconcave on a neighgorhood of
$\bar{U}.$ Thus,  for a  small 
$\epsilon>0,$ the functions $\etaup_{\zetaup}+\epsilon|z-z_0|^2,$ for $\phiup(\zetaup)=\phiup(p_0)-\deltaup^2,$
are still $r$-pseudoconcave on $U.$ Passing to the infimum, 
we deduce, 
by using \cite[Proposition 2.2. (ii)]{Dieu2006} that $\phiup+\epsilon |z-z_0|^2$ is weakly $r$-pseudoconcave
on $U.$ The proof is complete.
\end{proof}
From this and the remarks at the beginning of this subsection, we obtain:
\begin{thm}\label{vanishing3}
Let $M_0$ be a compact $\nt$-reductive 
homogeneous $CR$ manifold, with  canonical complex
 embedding $M_0\hookrightarrow{M}_-$.\par 
If  $M_0$ is an $r$-psudoconvave $CR$-manifold, then 
 $M_-$ is an $r$-pseudoconcave complex manifold and
for every coherent sheaf $\mathcal{F}$ we have
\begin{equation} \label{vanishingthx} 
\dim\big(\mathbf{H}^{j}(M_0,\mathcal{F})
\simeq\mathbf{H}^{j}(M_-,\mathcal{F})\big)
<\infty, 
\;\;\forall j<r-\mathrm{hd}(\mathcal{F}).
\end{equation}
In particular,
\begin{equation}\vspace{-20pt}
\dim\big(
\mathbf{H}^{p,j}(M_0)\simeq \mathbf{H}^{p,j}(M_-)
\big)<\infty,\;\;\forall j<r.
\end{equation}
\qed
\end{thm} 
\providecommand{\bysame}{\leavevmode\hbox to3em{\hrulefill}\thinspace}
\providecommand{\MR}{\relax\ifhmode\unskip\space\fi MR }
\providecommand{\MRhref}[2]{%
  \href{http://www.ams.org/mathscinet-getitem?mr=#1}{#2}
}
\providecommand{\href}[2]{#2}

\end{document}